\documentclass{amsart}
\usepackage{amssymb,graphicx}
\usepackage{mathrsfs}
\usepackage{enumerate}
\usepackage{a4wide}
\usepackage[colorlinks=true,citecolor=blue,linkcolor=blue]{hyperref}

\numberwithin{equation}{section}

\newtheorem{theorem}{Theorem}[subsection]
\newtheorem{proposition}[theorem]{Proposition}
\newtheorem{corollary}[theorem]{Corollary}
\newtheorem{definition}[theorem]{Definition}
\newtheorem{lemma}[theorem]{Lemma}

\theoremstyle{remark}

\newtheorem{remark}[theorem]{Remark}

\title{Lipschitz estimates in quasi-Banach Schatten ideals}
\date{\today}

\author{E. McDonald}
\author{F. Sukochev}

\begin{document}
\maketitle{}

\begin{abstract}
We study the class of functions $f$ on $\mathbb{R}$ satisfying a Lipschitz estimate in the Schatten ideal $\mathcal{L}_p$ for $0 < p \leq 1$.
The corresponding problem with $p\geq 1$ has been extensively studied, but the quasi-Banach range $0 < p < 1$ is by comparison poorly understood. Using techniques from wavelet analysis, we prove that Lipschitz functions belonging to the homogeneous Besov class $\dot{B}^{\frac{1}{p}}_{\frac{p}{1-p},p}(\mathbb{R})$ obey the estimate
$$
    \|f(A)-f(B)\|_{p} \leq C_{p}(\|f'\|_{L_{\infty}(\mathbb{R})}+\|f\|_{\dot{B}^{\frac{1}{p}}_{\frac{p}{1-p},p}(\mathbb{R})})\|A-B\|_{p}
$$
for all bounded self-adjoint operators $A$ and $B$ with $A-B\in \mathcal{L}_p$. In the case $p=1$, our methods recover and provide a new perspective on a result of Peller that $f \in \dot{B}^1_{\infty,1}$ is sufficient for a function to be Lipschitz in $\mathcal{L}_1$. 
We also provide related H\"older-type estimates, extending results of Aleksandrov and Peller.
In addition, we prove the surprising fact that non-constant periodic functions on $\mathbb{R}$
are not Lipschitz in $\mathcal{L}_p$ for any $0 < p < 1$. This gives counterexamples to a 1991 conjecture of Peller that $f \in \dot{B}^{1/p}_{\infty,p}(\mathbb{R})$ is sufficient for $f$
to be Lipschitz in $\mathcal{L}_p$.
\end{abstract}

\section{Introduction}
Let $H$ be a (complex and separable) Hilbert space, and denote
the operator norm by $\|\cdot\|_\infty$. A function $f:\mathbb{R}\to \mathbb{C}$
is said to be \emph{operator Lipschitz} if there exists a constant $C_f$ such that
\begin{equation*}
    \|f(A)-f(B)\|_\infty \leq C_f\|A-B\|_\infty,\quad A,B\in \mathcal{B}_{\mathrm{sa}}(H)
\end{equation*}
where $\mathcal{B}_{\mathrm{sa}}(H)$ denotes the space of all bounded self-adjoint linear operators on $H$. It has been
known since the work of Farforovskaya that not all Lipschitz
functions are operator Lipschitz \cite{Farforovskaja-1972} and it was later discovered that even the absolute value function $f(t) = |t|$ is not operator Lipschitz \cite{Kato1973,Davies-jlms-1988}. The problem of characterising the class
of operator Lipschitz functions has received considerable attention,
with early contributions from Daletskii and Krein \cite{Daletskii-Krein-1951,Daletskii-Krein-1956} and substantial advances by Birman, Solomyak \cite{Birman-Solomyak-I,Birman-Solomyak-II,Birman-Solomyak-III}, Aleksandrov, and Peller \cite{Aleksandrov-Peller-holder-2010,Aleksandrov-Peller-holder-zygmund-2010,Peller-besov-1990}. Some surveys on the topic are \cite{Aleksandrov-Peller-survey,Birman-Solomyak-2003,Peller-survey-2010,SkripkaTomskova}. 
At present no analytic condition on $f$ that is both necessary and sufficient for $f$ to be operator Lipschitz is known, however it has been proved by Peller that it is sufficient for $f$ to be Lipschitz and in the homogeneous Besov class $\dot{B}^{1}_{\infty,1}(\mathbb{R})$ 
\cite[Theorem 2]{Peller-besov-1990}. 
In other words, it suffices that $f$ be Lipschitz and
\begin{equation*}
    \int_0^\infty \sup_{t \in \mathbb{R}} |f(t+h)-2f(t)+f(t-h)|\frac{dh}{h^2} < \infty.
\end{equation*}    
Slightly weaker sufficient conditions are due to Arazy, Barton and Friedman \cite{Arazy-Barton-Friedman-1990}, \cite[Section 3.13]{Aleksandrov-Peller-survey}.

A more general problem which has also been of interest to many authors involves Lipschitz estimates
in operator ideals, the most important of which are the Schatten-von Neumann ideals.
For $0 < p < \infty$ the Schatten-von Neumann ideal $\mathcal{L}_p$ is defined as the class of operators $T$ on $H$ with $\|T\|_p := \mathrm{Tr}(|T|^p)^{1/p} < \infty$, where $\mathrm{Tr}$ is the operator trace. A function $f$ is said to be $\mathcal{L}_p$-operator Lipschitz if there is a constant $C_{f,p}>0$ such that
\begin{equation}\label{lipschitz_estimate}
    \|f(A)-f(B)\|_p \leq C_{f,p}\|A-B\|_p,\quad A,B\in \mathcal{B}_{\mathrm{sa}}(H),\;A-B\in \mathcal{L}_p.
\end{equation}
It is well-known that all Lipschitz functions are $\mathcal{L}_2$-operator Lipschitz, and that the class of $\mathcal{L}_1$-operator Lipschitz functions coincides with the class of
operator Lipschitz functions.
It is now known that if $1 < p <\infty$ then for \eqref{lipschitz_estimate} to hold it is necessary and sufficient that $f$ be Lipschitz \cite{PS-acta}. 
It is also known that Lipschitz functions satisfy a weak-type estimate in $\mathcal{L}_1$ \cite{CPSZ1,cpsz}.

By contrast, the range $0 < p < 1$ is poorly understood. The primary obstacle is that for this range of $p$, the ideal $\mathcal{L}_p$ is not a Banach space, but merely a quasi-Banach space. 
The failure of the triangle inequality makes many of the methods used in the $p \geq 1$ case inapplicable. For example, Peller's proof in \cite{Peller-besov-1990}
of the sufficiency of $f \in \dot{B}^1_{\infty,1}(\mathbb{R})$ for a Lipschitz function to be operator Lipschitz is based on a representation of $f(A)-f(B)$ as an operator-valued integral. Since $\mathcal{L}_p$ is a quasi-Banach space when $p<1$, the usual theories of Bochner or weak integration break down for $\mathcal{L}_p$-valued functions and it does not appear to be possible to adapt the proof of \cite[Theorem 2]{Peller-besov-1990} to the quasi-Banach case. Nonetheless, a number of important results for $0 < p < 1$ have been found by Rotfel'd \cite{Rotfeld1968,Rotfeld-trudy-1977}, Aleksandrov and Peller \cite{Peller-p-lt-1-1987,Aleksandrov-Peller-hankel-and-toeplitz-2002,Aleksandrov-Peller-acb-2020}, and Ricard \cite{Ricard-2018}.

For $0 < p < 1$ some results are known in the corresponding theory of operator Lipschitz functions of unitary operators. Peller has proved \cite[Theorem 1]{Peller-p-lt-1-1987} that if $0 < p \leq 1$ and $\phi \in B^{\frac{1}{p}}_{\infty,p}(\mathbb{T})$ (a Besov
space on the unit circle) then for all unitary operators $U$ and $V$ on $H$ with $U-V \in \mathcal{L}_p$ we have the inequality
\begin{equation}\label{circle_lipschitz}
    \|\phi(U)-\phi(V)\|_p \leq c_p\|\phi\|_{B^{\frac{1}{p}}_{\infty,p}(\mathbb{T})}\|U-V\|_p.
\end{equation}
Peller also proved \cite[Theorem 3]{Peller-p-lt-1-1987} that the condition $\phi \in B^{1/p}_{p,p}(\mathbb{T})$ is necessary for $\phi$ to satisfy a Lipschitz estimate of the form \eqref{circle_lipschitz} (but possibly with a different constant).

In 1991, Peller conjectured that a similar result holds for functions on $\mathbb{R}$, namely that if $f \in \dot{B}^{\frac{1}{p}}_{\infty,p}(\mathbb{R})$ then $f$
is $\mathcal{L}_p$-operator Lipschitz \cite[Page 14]{Peller-besov-1990}.

Via the Cayley transform, it is possible to directly infer sufficient conditions for a function $f:\mathbb{R}\to\mathbb{C}$ to satisfy \eqref{lipschitz_estimate} from \eqref{circle_lipschitz}.
However, we are not aware of any previous detailed study of $\mathcal{L}_p$-operator Lipschitz functions on $\mathbb{R}$ for $0 < p < 1.$ The following surprising example (proved in Section \ref{periodic_section}) is evidence that the 
theory is in fact very different from the case of functions on $\mathbb{T}$.
\begin{theorem}\label{periodic_failure}
    Let $f:\mathbb{R}\to\mathbb{C}$ be a non-constant periodic function. Then $f$ is not $\mathcal{L}_p$-operator Lipschitz for any $p \in (0,1)$. 
\end{theorem}    
Theorem \ref{periodic_failure} proves that even $C^\infty$ functions with all derivatives bounded are not necessarily $\mathcal{L}_p$-operator Lipschitz for any $0 < p < 1$. 
In particular, the function $h(t) = e^{it}$ is not $\mathcal{L}_p$-operator Lipschitz for any $p \in (0,1)$. This provides a counterexample to Peller's conjecture stated above, as $h$ belongs to the homogeneous Besov space $\dot{B}^{1/p}_{\infty,p}(\mathbb{R})$ for every $p \in (0,1)$.

The main result of this paper is the following theorem, which provides a general
sufficient condition for a function to be $\mathcal{L}_p$-operator Lipschitz naturally extending \cite[Theorem 2]{Peller-besov-1990}.
\begin{theorem}\label{main_sufficient_condition}
    Let $0 < p \leq 1$, and let $f$ be a Lipschitz function on $\mathbb{R}$ belonging to the homogeneous Besov space $\dot{B}^{\frac{1}{p}}_{\frac{p}{1-p},p}(\mathbb{R})$. There exists a constant $C_p > 0$ such that for 
    all bounded self-adjoint operators $A$ and $B$ with $A-B\in \mathcal{L}_p$ we have
    $$
        \|f(A)-f(B)\|_{p} \leq C_{p}(\|f'\|_{L_\infty(\mathbb{R})}+\|f\|_{\dot{B}^{\frac{1}{p}}_{\frac{p}{1-p},p}(\mathbb{R})})\|A-B\|_{p}.
    $$
    (Here, and throughout, we use the convention that $\frac{p}{1-p}=\infty$ when $p=1$.)
\end{theorem}
The assumption that $A$ and $B$ are bounded is made only for the convenience of exposition, and can very likely be removed. The constant $C_p$ does not depend on the operator norms of $A$
or $B$. Standard arguments show that Theorem \ref{main_sufficient_condition} also implies commutator estimates of the form
\begin{equation*}
    \|f(A)X-Xf(A)\|_{p} \leq C_p(\|f'\|_{\infty}+\|f\|_{\dot{B}^{\frac{1}{p}}_{\frac{p}{1-p},p}})\|AX-XA\|_p
\end{equation*}
and more generally quasi-commutator estimates of the form
\begin{equation*}
    \|f(A)X-Xf(B)\|_p \leq C_p(\|f'\|_\infty+\|f\|_{\dot{B}^{\frac{1}{p}}_{\frac{p}{1-p},p}})\|AX-XB\|_p.
\end{equation*}
See Section \ref{lp_lip_section} below for details.
    
In the case $p=1$, Theorem \ref{main_sufficient_condition} reduces to Peller's sufficient condition \cite[Theorem 2]{Peller-besov-1990}.
In his original proof, the function $f$ was decomposed into Littlewood-Paley components. However we have been unable to adapt these methods to $p<1$
due to the existence of non-$\mathcal{L}_p$-operator Lipschitz functions with compactly supported Fourier transform from Theorem \ref{periodic_failure}.

Our proof of Theorem \ref{main_sufficient_condition} is instead based on the decomposition of $f$ into wavelet series. Wavelet-based methods have had a considerable impact on harmonic analysis and approximation theory over the past three decades, however to our knowledge this is the first time that the wavelet decomposition has been applied in the study of operator Lipschitz functions. We note tangentially that wavelets were used by Peng in the related topic of integral multipliers \cite{Peng-wavelets-1993}, although otherwise this potentially very fruitful technique has yet to be exploited.

We do not discuss necessary conditions here, however a necessary condition for $f$ to be $\mathcal{L}_p$-operator Lipschitz for $0 < p \leq 1$ in terms of Hankel operators has been found by Peller, see \cite[Theorem 6]{Peller-besov-1990} for details.

A related theme is operator H\"older continuity, which has been studied extensively by Aleksandrov and Peller \cite{Aleksandrov-Peller-holder-2010,Aleksandrov-Peller-holder-zygmund-2010}.
In Sections \ref{holder_section} and \ref{weak_holder_section} we also prove a number of H\"older-type estimates, extending some of the results in \cite{Aleksandrov-Peller-holder-2010}.
For example, we prove that if $0 < \alpha <1$ and $0 < p \leq 1$ then for all $0 < \alpha<1,$ if $f$ is an $\alpha$-H\"older function belonging to $\dot{B}^{\alpha+\frac{1-p}{p}}_{\frac{p}{1-p},\infty}(\mathbb{R})$ we have
\begin{equation*}
    \|f(A)-f(B)\|_{\frac{p}{\alpha},\infty} \leq C_{p,\alpha}\|f\|_{\dot{B}^{\alpha+\frac{1-p}{p}}_{\frac{p}{1-p},\infty}(\mathbb{R})}\|A-B\|_{p}^{\alpha},\quad A,B\in \mathcal{B}_{\mathrm{sa}}(H),\; A-B\in \mathcal{L}_p.
\end{equation*}
Here, $\|\cdot\|_{\frac{p}{\alpha},\infty}$ is a weak Schatten quasi-norm.  This result extends \cite[Theorem 5.4]{Aleksandrov-Peller-holder-2010}, and coincides with that result for $p=1$. H\"older-type estimates of this nature
are related to those in \cite{HSZ2019,Ricard-2018}. In \cite{Ricard-2018} it was proved that for all $0 < \alpha < 1$ and $0 < p < \infty$ we have
\[
    \||A|^{\alpha}-|B|^{\alpha}\|_{\frac{p}{\alpha}} \lesssim \|A-B\|_p^{\alpha},\quad A,B \in \mathcal{B}_{\mathrm{sa}}(H),\; A-B \in \mathcal{L}_{p}.
\]
Since the function $f(t) = |t|^{\alpha}$ belongs to $\dot{B}^{\alpha+\frac{1-p}{p}}_{\frac{p}{1-p},\infty}(\mathbb{R})$, the results here imply a weaker estimate than \cite{Ricard-2018} but for a wider class of functions. We discuss this relationship
in more detail in Section \ref{weak_holder_section}.

We wish to extend our gratitude to Dmitriy Zanin for many helpful discussions and suggestions relating to this paper and to Jinghao Huang for his careful reading and comments. We also wish to express our gratitude to the two anonymous reviewers whose thoughtful comments improved this article.

\section{Preliminaries}
\subsection{Operator ideals and Schur multipliers}
We recall some details concerning Schatten ideals and Schur multipliers. Additional details on Schatten $\mathcal{L}_p$ spaces may be found in \cite{GohbergKrein,Simon-trace-ideals-2005}, and for further discussion of Schur multipliers see \cite{Birman-Solomyak-2003, Pisier-book-2001, SkripkaTomskova, Ricard-2018,Aleksandrov-Peller-hankel-and-toeplitz-2002,Aleksandrov-Peller-acb-2020}. Let $H$ be a Hilbert space. Denote by $\mathcal{B}(H)$ the algebra of all bounded linear operators on $H$, with operator norm denoted $\|\cdot\|_{\infty}$. Given a compact operator $T \in \mathcal{B}(H)$, denote by $\mu(T) = \{\mu(j,T)\}_{j=0}^\infty$ the sequence of singular values of $T$, which may be defined as
\begin{equation*}
    \mu(j,T) := \inf\{\|T-R\|_{\infty}\;:\; \mathrm{rank}(R)\leq j\}.
\end{equation*}
Equivalently, $\mu(T)$ is the sequence of eigenvalues of the absolute value $|T|$ arranged in non-increasing order with multiplicities.

For $0 < p < \infty$, denote by $\ell_p$ the space of $p$-summable sequences. The Schatten-von Neumann $\mathcal{L}_p$ space is the space of compact operators $T$ with singular value sequence in $\ell_p$. That is, $\mathcal{L}_p$
is the set of compact operators $T$ such that
\begin{equation*}
    \|T\|_p := \mathrm{Tr}(|A|^{p})^{\frac{1}{p}} = \left(\sum_{j=0}^\infty \mu(j,T)^p\right)^{1/p} = \|\mu(T)\|_{\ell_p} < \infty.
\end{equation*}
For $p \geq 1$, this defines a Banach norm on $\mathcal{L}_p$. For $0 < p < 1$, this is only a quasi-norm obeying the $p$-triangle inequality
\begin{equation*}
    \|T+S\|_p^p \leq \|T\|_p^p+\|S\|_p^p,\quad T,S \in \mathcal{L}_p.
\end{equation*}
See \cite[Proposition 6]{Kosaki-jfa-1984}, \cite{Rotfeld1968}.

We will also briefly refer to weak $L_p$-norms. For $p\in (0,\infty)$, the weak $L_p$-norm is defined by
\begin{equation}\label{weak_lp_def}
    \|T\|_{p,\infty} := \sup_{n\geq 0}\;(n+1)^{\frac{1}{p}}\mu(n,T).
\end{equation}

\subsection{Schur multipliers of $\mathcal{L}_p$}
Let $n\geq 1$. The Schur product of two matrices $A,B \in M_n(\mathbb{C})$ is defined as the entry-wise product
\begin{equation*}
    A\circ B = \{A_{j,k}B_{j,k}\}_{j,k=1}^n.
\end{equation*}
For $0 < p \leq 1$, the $\mathcal{L}_p$-bounded Schur multiplier norm of $A$ is defined as
\begin{equation*}
    \|A\|_{\mathrm{m}_p} := \sup_{\|B\|_p\leq 1} \|A\circ B\|_{p}.
\end{equation*}
Note that
\begin{equation}\label{algebra}
\|A\circ B\|_{\mathrm{m}_p} \leq \|A\|_{\mathrm{m}_p}\|B\|_{\mathrm{m}_p}.
\end{equation}
The $p$-subadditivity of the $\mathcal{L}_p$-quasi-norm readily implies that
\begin{equation}\label{p-prop}
    \|A+B\|_{\mathrm{m}_p}^p \leq \|A\|_{\mathrm{m}_p}^p+ \|B\|_{\mathrm{m}_p}^p.
\end{equation}

For $p \leq 1$, the $\mathrm{m}_p$-quasi-norm can be computed using rank one matrices. For $1\leq n\leq \infty$, we denote by $\ell_2^n$ either $\mathbb{C}^n$ if $n<\infty$ or $\ell_2(\mathbb{N})$ if $n=\infty$.
\begin{lemma}\label{rank_one_suffices}
    Let $A \in M_{n}(\mathbb{C})$ where $1\leq n\leq \infty$ and assume that $0 < p\leq 1$. Then
    \begin{equation*}
        \|A\|_{\mathrm{m}_p} := \sup_{\|\xi\|\leq 1,\, \|\eta\|\leq 1} \|A\circ (\xi\otimes \eta)\|_{p}.
    \end{equation*}
    The supremum here is over vectors $\xi, \eta$ in the unit ball of $\ell_2^n$, and $\xi\otimes \eta$ denotes the rank one matrix
    \begin{equation*}
        (\xi\otimes \eta)_{j,k} = \xi_j\eta_k,\quad 1\leq j,k\leq n
    \end{equation*}
    where $\xi_j$ and $\eta_k$ denote the $j$ and $k$th entries of $\xi$ and $\eta$ respectively.
\end{lemma}
\begin{proof}
    The matrix $\xi\otimes \eta$ is proportional to a rank one projection with constant equal to $\|\xi\|\|\eta\|$. Therefore,
    \begin{equation*}
        \|\xi\otimes \eta\|_p = \|\xi\|\|\eta\|\leq 1.
    \end{equation*} 
    It follows that
    \begin{equation*}
        \sup_{\|\xi\|,\|\eta\|\leq 1} \|A\circ (\xi\otimes \eta)\|_p \leq \|A\|_{\mathrm{m}_p}.
    \end{equation*}
    Using the Schmidt decomposition, if $B \in M_{n}(\mathbb{C})$ then there exist sequences $\{\xi_j\}_{j=0}^{n-1}$, $\{\eta_j\}_{j=0}^{n-1}$
    of unit vectors in $\ell_2^n$ such that 
    \begin{equation*}
        B = \sum_{j=0}^{n-1} \mu(j,B)\xi_j\otimes \eta_j.
    \end{equation*}
    By the $p$-subadditivity of the $\mathcal{L}_p$-quasi-norm, we have
    \begin{equation*}
        \|A\circ B\|_p^p \leq \sum_{j=0}^{n-1} \mu(j,B)^p\|A\circ (\xi_j\otimes \eta_j)\|_p^p \leq \left(\sum_{j=0}^{n-1} \mu(j,B)^p\right)\sup_{\|\xi\|,\|\eta\|\leq 1} \|A\circ (\xi\otimes \eta)\|_p^p.
    \end{equation*} 
    By definition, $\sum_{j=0}^{n-1} \mu(j,B)^p = \|B\|_p^p$. Therefore
    \begin{equation*}
        \|A\|_{\mathrm{m}_p} \leq \sup_{\|\xi\|,\|\eta\|\leq 1} \|A\circ (\xi\otimes \eta)\|_p.
    \end{equation*}
\end{proof}

A property of the $\mathrm{m}_p$-Schur norm that we shall use frequently is the following:
\begin{lemma}
    Let $A\in M_n(\mathbb{C})$, where $1\leq n\leq \infty$. If $A'$ is a submatrix of $A$, then
    \begin{equation*}
        \|A'\|_{\mathrm{m}_p}\leq \|A\|_{\mathrm{m}_p}.
    \end{equation*}
\end{lemma}
It follows that adding rows or columns to a matrix cannot decrease the $\mathrm{m}_p$-norm.

Following the notation of Aleksandrov and Peller \cite{Aleksandrov-Peller-hankel-and-toeplitz-2002}, for $0 < p \leq 1$ a conjugate index $p^\sharp$
is defined by
\begin{equation*}
    p^\sharp := \begin{cases} 
                    \frac{p}{1-p},\quad p < 1,\\
                    \infty,\quad p=1.
                \end{cases}
\end{equation*}
That is, $p^{\sharp}$ is the unique element of $(0,\infty]$ such that
\begin{equation*}
    \frac{1}{p} = \frac{1}{p^{\sharp}}+1.
\end{equation*}
As an application of the H\"older inequality for Schatten ideals \cite[Property 2, page 92]{GohbergKrein}, \cite[Lemma 1]{Kosaki-jfa-1984}, it follows that
\begin{equation}\label{operator_holder_inequality}
    \|AB\|_p \leq \|A\|_{p^{\sharp}}\|B\|_1,\quad 0 < p \leq 1.
\end{equation}
Therefore, Lemma \ref{rank_one_suffices} implies that
\begin{equation}\label{psharp_upper_bound}
    \|A\|_{\mathrm{m}_p} = \sup_{\|\xi\|,\|\eta\|\leq 1} \|A\circ (\xi\otimes \eta)\|_p \leq \|A\|_{p^{\sharp}}.
\end{equation}
(c.f. \cite[Theorem 3.1]{Aleksandrov-Peller-hankel-and-toeplitz-2002}). 
There also holds the H\"older inequality for sequences
\begin{equation}\label{sequential_holder}
    \|xy\|_{\ell_p} \leq \|x\|_{\ell_{p^\sharp}}\|y\|_{\ell_1},\quad x \in \ell_{p^\sharp},\, y \in \ell_1.
\end{equation}

The next lemma is a very slight modification of \cite[Theorem 3.2]{Aleksandrov-Peller-hankel-and-toeplitz-2002}.
For $1\leq n < \infty$, denote by $\{e_j\}_{j=0}^{n-1}$ the canonical basis of the $n$-dimensional vector space $\ell_2^n$. 
The matrix basis of $M_n(\mathbb{C})$ shall be denoted $\{e_j\otimes e_k\}_{j,k=0}^{n-1}$. A matrix $X \in M_n(\mathbb{C})$ is said to be diagonal
with respect to $\{e_j\otimes e_k\}_{j,k=0}^{n-1}$ when $\langle e_j,Xe_k\rangle = 0$ for $j\neq k$.

\begin{lemma}\label{block_diagonal_formula}
    Let $A\in M_n(\mathbb{C})$, $1\leq n\leq \infty$ be a matrix having a generalised block diagonal structure in the following sense:
    There exist pairwise orthogonal projections $\{p_j\}_{j=1}^N$ and $\{q_j\}_{j=1}^N$, diagonal with respect to the matrix basis $\{e_j\otimes e_k\}_{j,k=1}^n$
    such that $\sum_{j=1}^N q_j = \sum_{j=1}^N p_j = 1$ and
    \begin{equation*}
        A = \sum_{j=1}^N p_jAq_j.
    \end{equation*}
    It follows that for $0 < p < 1$ we have
    \begin{equation*}
        \|A\|_{\mathrm{m}_p} \leq \left(\sum_{j=1}^N \|p_jAq_j\|_{\mathrm{m}_p}^{p^\sharp}\right)^{\frac{1}{p^\sharp}}.
    \end{equation*}
    and
    \begin{equation*}
        \|A\|_{\mathrm{m}_1} \leq \max_{1\leq j\leq N} \|p_jAq_j\|_{\mathrm{m}_1}.
    \end{equation*} 
\end{lemma}

We also define the $\mathrm{m}_p$-Schur norm for matrices indexed by arbitrary, possibly infinite and uncountable sets.
\begin{definition}
    If $A = \{A_{t,s}\}_{t,s\in T\times S}$ is an infinite matrix indexed by sets $T$ and $S$, we define
    \begin{equation*}
        \|A\|_{\mathrm{m}_p} := \sup_{|T_0|< \infty\,|S_0|< \infty} \|\{A_{t,s}\}_{t,s\in T_0\times S_0}\|_{\mathrm{m}_p}.
    \end{equation*}
    That is, the $\mathrm{m}_p$-norm of an infinite matrix is defined as the supremum of the $\mathrm{m}_p$-norms of all finite submatrices.
    If $\|A\|_{\mathrm{m}_p} < \infty$, then the matrix $A$ is said to be an $\mathcal{L}_p$-bounded Schur multiplier.
\end{definition}   

The analogy of Lemma \ref{block_diagonal_formula} holds for matrices indexed by arbitrary sets, and also
note that the analogy of \eqref{p-prop} holds. That is,
\begin{equation*}
    \|A+B\|_{\mathrm{m}_p}^p \leq \|A\|_{\mathrm{m}_p}^p +\|B\|_{\mathrm{m}_p}^p
\end{equation*}
whenever $A$ and $B$ are matrices indexed by the same sets.

\subsection{Besov spaces}\label{besov_section}
Denote by $\mathcal{S}(\mathbb{R})$ the algebra of all Schwartz class functions on $\mathbb{R}$, with its canonical Fr\'echet topology, and denote by $\mathcal{S}'(\mathbb{R})$ its topological
dual, the space of tempered distributions. Let $\Phi$ be a smooth function on $\mathbb{R}$ supported in the set
\begin{equation*}
    [-2,-1+\frac{1}{7})\cup (1-\frac{1}{7},2],
\end{equation*}
and identically equal to $1$ in the set $[-2+\frac{2}{7},-1)\cup (1,2-\frac{2}{7}].$
We assume that
\begin{equation*}
    \sum_{n\in \mathbb{Z}} \Phi(2^{-n}\xi) = 1\quad \xi\neq 0.
\end{equation*}

We will use a homogeneous Littlewood-Paley decomposition $\{\Delta_n\}_{n\in \mathbb{Z}}$ where
$\Delta_n$ is the operator on $\mathcal{S}'(\mathbb{R})$ of Fourier multiplication by the function $\xi\mapsto \Phi(2^{-n}\xi)$. 

For $s \in \mathbb{R}$ and $p,q\in (0,\infty]$ we consider the homogeneous Besov space $\dot{B}^s_{p,q}(\mathbb{R})$. We refer to \cite{Sawano2018,Triebel-1} for comprehensive accounts of the theory of Besov spaces. 
In terms of the Littlewood-Paley decomposition $\{\Delta_j\}_{j\in \mathbb{Z}}$, a distribution $f \in \mathcal{S}'(\mathbb{R})$ is said to belong to the homogeneous Besov space
$\dot{B}^s_{p,q}(\mathbb{R})$, where $s \in \mathbb{R}$ and $p,q\in (0,\infty]$ if
\begin{equation}\label{besov_seminorm_def}
    \|f\|_{\dot{B}^s_{p,q}} := \|\{2^{js}\|\Delta_j f\|_p\}_{j\in \mathbb{Z}}\|_{\ell_q(\mathbb{Z})} < \infty.
\end{equation}
This definition follows \cite[Section 2.4]{Sawano2018}, \cite[Section 2.2.1]{Grafakos-2}, \cite[Section 5.1.3]{Triebel-1}.
This is only a seminorm, and $\|f\|_{\dot{B}^s_{p,q}} = 0$ for all polynomials $f$. The homogeneous Besov space is distinguished from the inhomogeneous Besov space $B^s_{p,q}(\mathbb{R})$,
which will not play a role in the present paper. 

Note that if $f \in \dot{B}^s_{p,q}(\mathbb{R})$ it will not necessarily be the case that there is an equality of distributions
\begin{equation*}
    f = \sum_{n\in \mathbb{Z}} \Delta_n f.
\end{equation*}
For example, if $f$ is a polynomial then the above right hand side is zero. However, there is an equality $f = \sum_{n\in \mathbb{Z}}\Delta_n f$
in the space of distributions modulo polynomials of degree at most $L>s-\frac{1}{p}$, see \cite[Theorem 2.31]{Sawano2018}.

In \cite{Peller-besov-1990, Aleksandrov-Peller-survey}, a slight modification of the definition of $\dot{B}^1_{\infty,1}(\mathbb{R})$
was made, and therefore in order to properly compare our results we explain how our present conventions align with those in \cite{Aleksandrov-Peller-survey}.

Say that that a distribution $f$ belongs to the modified homogeneous Besov space $\dot{B}^1_{\infty,1,\mathrm{mod}}(\mathbb{R})$ if $f \in \dot{B}^1_{\infty,1}(\mathbb{R})$
and the derivative $f'$ is expressed as
\begin{equation*}
    f' = \sum_{n\in \mathbb{Z}} (\Delta_nf)'
\end{equation*}
where the series converges in the sense of distributions. 

We now recall the well-known relation between $\dot{B}^1_{\infty,1,\mathrm{mod}}(\mathbb{R})$ and $\dot{B}^1_{\infty,1}(\mathbb{R})$. Recall that if $f$ is Lipschitz continuous, then $f$
is almost everywhere differentiable and $f'\in L_\infty(\mathbb{R})$ with $\|f'\|_{\infty}\leq \|f\|_{\mathrm{Lip}(\mathbb{R})}$ where $\|\cdot\|_{\mathrm{Lip}(\mathbb{R})}$
is the Lipschitz seminorm \cite[Subsection 3.1.6]{Federer-1969}. The reverse implication holds under the assumption that $f$ is absolutely continuous \cite[Corollary 2.9.20]{Federer-1969}.
\begin{lemma}
    Suppose that $f$ belongs to the modified homogeneous Besov space $\dot{B}^1_{\infty,1,\mathrm{mod}}(\mathbb{R})$. Then $f$ is Lipschitz continuous.
    
    Conversely, if $f$ is a Lipschitz function belonging to $\dot{B}^1_{\infty,1}(\mathbb{R})$, then there exists a constant $c$ with $|c|\leq \|f'\|_\infty +\|f\|_{\dot{B}^1_{\infty,1}}$ such that $f(t)-ct\in \dot{B}^1_{\infty,1, \mathrm{mod}}(\mathbb{R})$.
\end{lemma}

\subsection{$\mathcal{L}_p$-operator Lipschitz functions}\label{lp_lip_section}
Let $f$ be a Lipschitz function on $\mathbb{R}$, and let $0 < p \leq \infty$. The following assertions are equivalent:
\begin{enumerate}[{\rm (i)}]
    \item{}\label{lipschitz_definition} There is a constant $c_f$ such that $\|f(A)-f(B)\|_p\leq c_f\|A-B\|_{p}$ for all bounded self-adjoint $A$ and $B$ with $A-B \in \mathcal{L}_p$,
    \item{}\label{commutator_lipschitz} There is a constant $c_f'$ such that $\|[f(A),X]\|_{p} \leq c_f'\|[A,X]\|_{p}$ for all bounded self-adjoint $A$ and bounded $X$ with $[A,X] \in \mathcal{L}_p$,
    \item{}\label{quasicommutator_lipschitz} There is a constant $c_f''$ such that $\|f(A)X-Xf(B)\|_p \leq c_f''\|AX-XB\|_p$ for all bounded self-adjoint $A,B$ and bounded $X$ with $AX-XB \in \mathcal{L}_p$,
    \item{}\label{schur_boundedness} The matrix of divided differences $\{f^{[1]}(t,s)\}_{t,s\in \mathbb{R}}$, where $f^{[1]}$ is defined as
    \begin{equation*}
        f^{[1]}(t,s) := \frac{f(t)-f(s)}{t-s},\quad t\neq s\in \mathbb{R}.
    \end{equation*}
    is a Schur multiplier of $\mathcal{L}_p$ in the sense that
    \begin{equation*}
        \sup_{\lambda,\mu} \|\{f^{[1]}(\lambda_j,\mu_k)\}_{j,k=0}^n\|_{\mathrm{m}_p} < \infty
    \end{equation*}
    where the supremum ranges over all \emph{disjoint} sequences $\lambda,\mu\subset \mathbb{R}$ and all $n\geq 1$.
\end{enumerate}
Note that the constants in each case might differ. The Schur multiplier condition in \eqref{schur_boundedness} is implied by the formally stronger assertion that $\|f^{[1]}\|_{\mathrm{m}_p}<\infty$.
For $p\geq 1$, this result has been proved in different contexts and at varying levels of generality in several places \cite[Theorem 10.1]{Aleksandrov-Peller-holder-zygmund-2010}, \cite[Corollary 5.6]{Kissin-Shulman-2005}, \cite[Theorem 3.1.1]{Aleksandrov-Peller-survey}, \cite[Theorem 3.4]{DDdPS-jfa-1997}, \cite[Lemma 2.4]{RicardArchBasel2015}.

While this fact is well-established when $\mathcal{L}_p$ is a Banach space; we are not aware of any published proof of the precisely the same assertions when $\|\cdot\|_{p}$ is merely a quasi-norm although we note that closely related statements have appeared in \cite[Section 7]{HSZ2019} and \cite{RicardArchBasel2015}. Nonetheless, the results when $p<1$ may be proved in the same way, following without any changes the proofs in \cite{DDdPS-jfa-1997}. Therefore we only state without proof the relevant implications.

The condition \eqref{schur_boundedness} may seem unfamiliar, since we only require that $\|\{f^{[1]}(\lambda_j,\mu_k)\}_{j,k=0}^n\|_{\mathrm{m}_p}$ be uniformly bounded over
all disjoint sequences $\lambda$ and $\mu$, rather than all sequences. This issue is irrelevant in the Banach case $p\geq 1$, since the diagonal matrix $\{\chi_{t=s}\}_{t,s\in \mathbb{R}}$ is an $\mathcal{L}_p$-bounded Schur multiplier
for all $p\geq 1$. This is false when $p<1$, and hence some caution is needed.

Of course, \eqref{quasicommutator_lipschitz} implies both \eqref{lipschitz_definition} and \eqref{commutator_lipschitz}. It is also the case that \eqref{commutator_lipschitz} implies \eqref{quasicommutator_lipschitz}; this follows by substituting for $A$ and $X$ the matrices
\begin{equation*}
    \widetilde{A} := \begin{pmatrix} A & 0 \\ 0 & B \end{pmatrix},\quad \widetilde{X} = \begin{pmatrix} 0 & X \\ X & 0\end{pmatrix}
\end{equation*}
and using the formula
\begin{equation*}
    [\widetilde{A},\widetilde{X}] = \begin{pmatrix} 0 & AX-XB\\ AX-XB & 0 \end{pmatrix}
\end{equation*}
so that \eqref{commutator_lipschitz} implies \eqref{quasicommutator_lipschitz} due to the unitary invariance of the $\mathcal{L}_p$-quasi-norm.

The following Lemma states that \eqref{lipschitz_definition} implies \eqref{commutator_lipschitz}.
\begin{lemma}\label{lipschitz_implies_commutator_lipschitz}
    Let $0 < p\leq 1$. If $f$ is a Borel function on $\mathbb{R}$ such that for all bounded self-adjoint operators $A$ and $B$ on $H$ with $A-B \in \mathcal{L}_p$ we have
    \begin{equation*}
        \|f(A)-f(B)\|_{p} \leq c_f\|A-B\|_{p}
    \end{equation*}
    for some constant $c_f$ not depending on $A$ or $B$, then for all self-adjoint operators $A=A^*$ bounded operators $X$ such that $[A,X] \in \mathcal{L}_p$ we have
    \begin{equation*}
        \|[f(A),X]\|_{p} \leq c_f\|[A,X]\|_{p}.
    \end{equation*}
\end{lemma}

The following lemma essentially states that \eqref{lipschitz_definition} implies \eqref{schur_boundedness}.
\begin{lemma}\label{lp_lip_implies_lp_schur}
    Let $0 < p \leq 1$. Suppose that $f:\mathbb{R}\to \mathbb{C}$ is a Borel function which is $\mathcal{L}_p$-operator Lipschitz. Then $f^{[1]}$ is an $\mathcal{L}_p$-bounded Schur multiplier in the sense of \eqref{schur_boundedness}.
\end{lemma}

The well-known converse result, which is that \eqref{schur_boundedness} implies \eqref{lipschitz_definition}, is as follows.
\begin{theorem}
    Let $0 < p \leq 1$. Let $f:\mathbb{R}\to \mathbb{C}$ be a Borel function such that $\{f^{[1]}(t,s)\}_{t,s \in \mathbb{R}}$ is an $\mathcal{L}_p$-bounded Schur multiplier in the sense of \eqref{schur_boundedness}. Then $f$ is $\mathcal{L}_p$-operator Lipschitz.
\end{theorem}

\section{Negative results}

\subsection{Periodic functions}\label{periodic_section}
We now prove Theorem \ref{periodic_failure}. The proof is based on negating Lemma \ref{lp_lip_implies_lp_schur} by selecting appropriate sequences
such that the matrix $\Gamma$ constructed in the proof of Lemma \ref{lp_lip_implies_lp_schur} is not an $\mathcal{L}_p$-bounded Schur multiplier. The specific form of $\Gamma$ will be a Toeplitz matrix, and necessary and sufficient conditions for a Toeplitz matrix to be an $\mathcal{L}_p$-bounded Schur multiplier are known \cite[Theorem 5.1]{Aleksandrov-Peller-hankel-and-toeplitz-2002}. However, for the sake of being self-contained we present an elementary argument.
\begin{lemma}\label{toeplitz_is_not_schur}
    Let $\varepsilon \in  (0,1)$, and let $T$ be the matrix
    \begin{equation*}
        T = \left\{\frac{1}{\varepsilon+j-k}\right\}_{j,k\geq 0}.
    \end{equation*}
    Then $T$ is not an $\mathcal{L}_p$-bounded Schur multiplier for any $p \in (0,1)$.
\end{lemma}
\begin{proof}
    Let $n,m \geq 1$, and consider the matrix $X_{n,m}$ defined as
    \begin{equation*}
        X_{n,m} = \sum_{j,k=0}^{n-1} e_{mj}\otimes e_{mk}.
    \end{equation*}
    Then $X$ is $n$ times a rank one projection, so,
    \begin{equation*}
        \|X_{n,m}\|_{p} = n.
    \end{equation*}
    We also have,
    \begin{equation*}
        T\circ X_{n,m} = \sum_{j,k=0}^{n-1} \frac{1}{\varepsilon+m(j-k)}e_{mj}\otimes e_{mk}.
    \end{equation*}
    Thus if $T$ is an $\mathcal{L}_p$-bounded Schur multiplier, then there is a constant $C>0$ such that for all $n,m\geq 1$ we have
    \begin{align*}
        \|T\circ X_{n,m}\|_p &= \left\|\sum_{j,k=0}^{n-1} \frac{1}{\varepsilon+m(j-k)}e_{mj}\otimes e_{mk}\right\|_p\\
                                &= \left\|\sum_{j,k=0}^{n-1} \frac{1}{\varepsilon+m(j-k)}e_{j}\otimes e_{k}\right\|_p\\
                                &\leq Cn.
    \end{align*}
    That is, for every $n,m\geq 1$ we have
    \begin{equation*}
        \left\|\sum_{j,k=0}^{n-1} \frac{1}{\varepsilon+m(j-k)}e_{j}\otimes e_{k}\right\|_p \leq Cn.
    \end{equation*}
    Taking the limit $m\to \infty$, the off diagonal terms vanish, leaving only the diagonal. This leads to 
    \begin{equation*}
        \|\sum_{j=0}^{n-1} \frac{1}{\varepsilon}e_{j}\otimes e_{j}\|_p =  \lim_{m\to\infty} \left\|\sum_{j,k=0}^{n-1} \frac{1}{\varepsilon+m(j-k)}e_{j}\otimes e_k\right\|_p \leq Cn.
    \end{equation*}
    The left hand side is equal to $n^{1/p}/\varepsilon$, and therefore
    \begin{equation*}
        n^{1/p-1} \leq C\varepsilon
    \end{equation*}
    for all $n\geq 1$, which is impossible since $p < 1$.
\end{proof}

\begin{remark}
    The result of \cite[Theorem 5.1]{Aleksandrov-Peller-hankel-and-toeplitz-2002} states that if $0 < p < 1$, then a Toeplitz matrix $\{t_{j-k}\}_{j,k\geq 0}$ is
    a Schur multiplier of $\mathcal{L}_p$ if and only if $\{t_n\}_{n\in \mathbb{Z}}$ is the sequence of Fourier coefficients of a $p$-convex combination of point masses
    on $\mathbb{T}$. In particular, $\{t_n\}_{n\in \mathbb{Z}}$ must be the sequence of Fourier coefficients of a singular measure. In the case of the matrix $T$ in Lemma \ref{toeplitz_is_not_schur}, we have $t_n = \frac{1}{\varepsilon+n},\; n\in \mathbb{Z}$, which is the sequence of Fourier coefficients of an $L_2$-function.
    It follows that $T$ is not an $\mathcal{L}_p$-bounded Schur multiplier and this amounts to an alternative proof of Lemma \ref{toeplitz_is_not_schur}. 
\end{remark}

Recall that if $f:\mathbb{R}\to \mathbb{C}$, we denote by $f^{[1]}$ the function on $\mathbb{R}^2\setminus \{(t,t)\;:\;t\in \mathbb{R}\}$.
\begin{equation*}
    f^{[1]}(t,s) = \begin{cases}
                        \frac{f(t)-f(s)}{t-s},\quad t \neq s,\\
                    \end{cases}
\end{equation*}
\begin{theorem}\label{compact_theorem}
    Let $f:\mathbb{R}\to \mathbb{C}$ be a non-constant periodic function.
    Then the infinite matrix $\{f^{[1]}(t,s)\}_{t,s \in \mathbb{R}}$ is not an $\mathcal{L}_p$-bounded Schur multiplier for any $p \in (0,1)$ in the sense of \eqref{schur_boundedness}. That is,
    \begin{equation*}
        \sup_{\lambda\cap\mu = \emptyset} \|\{f^{[1]}(\lambda_j,\mu_k)\}_{j,k=0}^n\|_{\mathrm{m}_p} = \infty.
    \end{equation*}
\end{theorem}
\begin{proof}
    By rescaling $f$ if necessary, we may assume without loss of generality that $f$ is $1$-periodic, and since $f$ is not constant we can select some $\varepsilon \in (0,1)$ such that $f(\varepsilon)\neq f(0)$.
    Consider the following two sequences:
    \begin{equation*}
        \lambda_j = j+\varepsilon,\quad \mu_k = k.
    \end{equation*}
    Due to $f$ being $1$-periodic, we compute $f^{[1]}(\lambda_j,\mu_k)$ as
    \begin{equation*}
        \frac{f(\lambda_j)-f(\mu_k)}{\lambda_j-\mu_k} = \frac{f(\varepsilon)-f(0)}{\varepsilon+j-k} = \frac{1}{\varepsilon+j-k}(f(\varepsilon)-f(0)).
    \end{equation*}
    Since $f(\varepsilon)-f(0) \neq 0$, it follows that
    \begin{equation*}
        \frac{1}{\varepsilon+j-k} = \frac{f(\lambda_j)-f(\mu_k)}{\lambda_j-\mu_k} \cdot \frac{1}{f(\varepsilon)-f(0)}.
    \end{equation*}
    It follows that if $\{f^{[1]}(t,s)\}_{t,s \in \mathbb{R}}$ were an $\mathcal{L}_p$-bounded Schur multiplier, then the matrix $\{\frac{1}{\varepsilon+j-k}\}_{j,k\geq 0}$ would also be an $\mathcal{L}_p$-bounded Schur multiplier, but this is false due to Lemma \ref{toeplitz_is_not_schur}.
\end{proof}    

Theorem \ref{compact_theorem}, combined with Lemma \ref{lp_lip_implies_lp_schur}, implies Theorem \ref{periodic_failure}.

\section{Positive results}
\subsection{Wavelet analysis}
A \emph{wavelet} is a function $\phi \in L_2(\mathbb{R})$ such that the family
\begin{equation*}
    \phi_{j,k}(t) = 2^{\frac{j}{2}}\phi(2^jt-k),\quad j,k\in \mathbb{Z},\quad t \in \mathbb{R}
\end{equation*}
of translations and dilations of $\phi$ forms an orthonormal basis of $L_2(\mathbb{R})$ \cite[Definition 6.6.1]{Grafakos-1}. For example, the Haar function
\begin{equation*}
    h(t) = \chi_{[0,1/2]}(t)-\chi_{(1/2,1]}(t),\quad t \in \mathbb{R}
\end{equation*}
is a wavelet.     
It is a theorem of Daubechies that there exist compactly supported $C^r$-wavelets for every $r > 0$ \cite{Daubechies-wavelets-1988}, \cite[Theorem 3.8.3]{Meyer-wavelets-1992}.

For this subsection, we will fix a compactly  supported wavelet $\psi$ of regularity $C^r$ for some $r>1$. In later subsections we will ask for additional smoothness on $\psi$. 
Every $f\in L_2(\mathbb{R})$ admits an $L_2$-convergent wavelet decomposition
\begin{equation*}
    f = \sum_{j,k\in \mathbb{Z}} \psi_{j,k}\langle f,\psi_{j,k}\rangle.
\end{equation*}
This is called a wavelet series. For brevity, denote
\begin{equation*}
    f_j = \sum_{k\in \mathbb{Z}} \psi_{j,k}\langle f,\psi_{j,k}\rangle\in L_2(\mathbb{R}),\quad j \in \mathbb{Z}
\end{equation*}
That is, we have the $L_2$-convergent series
\begin{equation*}
    f = \sum_{j\in \mathbb{Z}} f_j,\quad f_j(t) = \sum_{k\in \mathbb{Z}} 2^{\frac{j}{2}}\psi(2^jt-k)\langle f,\psi_{j,k}\rangle,\quad t \in \mathbb{R}.
\end{equation*}
Roughly speaking, our strategy will be to bound $\|f^{[1]}\|_{\mathrm{m}_p}$ using the wavelet decomposition and \eqref{p-prop} as follows
\begin{equation*}
    \|f^{[1]}\|_{\mathrm{m}_p}^p \leq \sum_{j\in \mathbb{Z}} \|f_j^{[1]}\|_{\mathrm{m}_p}^p.
\end{equation*}

Note that for arbitrary locally integrable functions $f$ on $\mathbb{R}$, the wavelet coefficient $\langle f,\psi_{j,k}\rangle$ is meaningful due to our assumption that $\psi$ is continuous
and compactly supported. It follows that for all locally integrable $f$, we can define
\begin{equation}\label{f_j_definition}
    f_j = \sum_{k\in \mathbb{Z}} \psi_{j,k}\langle f,\psi_{j,k}\rangle,\quad j\in \mathbb{Z}
\end{equation}
where the sum is finite on compact sets.

The following is \cite[Theorem 3.9.2]{Cohen2003}. We use the symbol $\approx$ to denote equivalence up to constants depending only on $p$ and the choice of wavelet.
\begin{lemma}\label{disjointifying}
    Let $\phi$ be an arbitrary wavelet on $\mathbb{R}$, and let $\alpha = \{\alpha_k\}_{k\in \mathbb{Z}}$ be a scalar sequence. Define
    \begin{equation*}
        \phi_{\alpha}(t) = \sum_{k\in \mathbb{Z}}\alpha_k\phi(t-k),\quad t \in \mathbb{R}.
    \end{equation*} 
    Then for all $p \in (0,\infty]$ such that $\phi$ is $p$-integrable we have
    \begin{equation*}
        \|\phi_{\alpha}\|_p \approx \|\alpha\|_{\ell_p}.
    \end{equation*}
\end{lemma}
Lemma \ref{disjointifying} relies on the fact that the family of translates $\{\phi(\cdot-j)\}_{j\in \mathbb{Z}}$ is locally linearly independent, which holds in particular when $\phi$ is a wavelet, and is false if $\phi$ were an arbitrary compactly supported function.

A simple consequence is the following identity for the $L_p$-norm of $f_j$, which is well-known. We provide a proof for convenience.
See \cite[Proposition 6.10.7]{Meyer-wavelets-1992} for a proof in the $p\geq 1$ case.
\begin{lemma}\label{wavelet_bernstein}
    Let $f$ be a locally integrable function. For every $p \in (0,\infty]$ and $j \in \mathbb{Z}$ we have 
    \begin{equation}\label{disjoint_supports}
        \|f_j\|_p \approx 2^{j\left(\frac{1}{2}-\frac{1}{p}\right)} \left(\sum_{k\in \mathbb{Z}} |\langle f,\psi_{j,k}\rangle|^p \right)^{1/p}.
    \end{equation}
    In particular, the sequence $\{\langle f,\psi_{j,k}\rangle\}_{k\in \mathbb{Z}}$ is $p$-summable if and only if $f_j \in L_p(\mathbb{R})$.
\end{lemma}
\begin{proof}
    We have
    \begin{equation*}
        f_j(t) = \sum_{k\in \mathbb{Z}} 2^{\frac{j}{2}}\psi(2^jt-k)\langle f,\psi_{j,k}\rangle,\quad t \in \mathbb{R}.
    \end{equation*}
    Therefore
    \begin{equation*}
        2^{-\frac{j}{2}}f_j(2^{-j}t)  =\sum_{k\in \mathbb{Z}} \psi(t-k)\langle f,\psi_{j,k}\rangle.
    \end{equation*}
    Applying Lemma \ref{disjointifying} with $\alpha = \{\langle f,\psi_{j,k}\rangle\}_{k\in \mathbb{Z}}$ implies that
    \begin{equation*}
        \|2^{-\frac{j}{2}}f_j(2^{-j}\cdot)\|_p \approx \left(\sum_{k\in \mathbb{Z}} |\langle f,\psi_{j,k}\rangle|^p\right)^{1/p}.
    \end{equation*}
    Using the rule $\|f(\lambda\cdot)\|_p = \lambda^{-\frac{1}{p}}\|f\|_p$, it follows that
    \begin{equation*}
        2^{-\frac{j}{2}+\frac{j}{p}}\|f_j\|_p \approx \left(\sum_{k\in \mathbb{Z}} |\langle f,\psi_{j,k}\rangle|^p\right)^{1/p}.
    \end{equation*}
\end{proof}

We note for future reference that since for $p \leq q$ and all locally integrable $f$ there holds the inequality
\begin{equation*}
    \left(\sum_{k\in \mathbb{Z}} |\langle f,\psi_{j,k}\rangle|^q\right)^{1/q} \leq \left(\sum_{k\in \mathbb{Z}} |\langle f,\psi_{j,k}\rangle|^p\right)^{\frac{1}{p}}
\end{equation*}
it follows from Lemma \ref{wavelet_bernstein} that for all $j\in \mathbb{Z}$,
\begin{equation}\label{sequential_bernstein}
    \|f_j\|_q \lesssim 2^{j(\frac{1}{p}-\frac{1}{q})}\|f_j\|_p,\quad p \leq q,
\end{equation}
The same holds for $q=\infty$. That is, $\|f_j\|_\infty \lesssim 2^{\frac{j}{p}}\|f_j\|_p$ for all $p < \infty$.

Besov spaces have very simple characterisations in terms of coefficients of wavelet series. The following is \cite[Theorem 7.20]{FrazierJawerthWeiss1991}. Related results in the inhomogeneous case are \cite[Theorem 3.7.7]{Cohen2003}, \cite[Theorem 4.7]{Sawano2018}, \cite[Theorem 1.20]{Triebel-4}  (see also \cite[Section 6.10]{Meyer-wavelets-1992} for $p,q\in [1,\infty]$).
\begin{theorem}\label{besov_space_wavelet_characterisation}
    Let $p,q \in (0,\infty]$ and $s \in \mathbb{R}$. Let $f$ be a locally integrable function, and let $\psi$ be a compactly supported $C^r$ wavelet for $r > |s|$. Then $f$ belongs to the homogeneous Besov class $\dot{B}^s_{p,q}(\mathbb{R})$ if and only if 
    \begin{equation*}
        \|f\|_{\dot{B}^s_{p,q}} \approx \left(\sum_{j\in \mathbb{Z}} 2^{jsq}\|f_j\|_p^q\right)^{1/q} < \infty.
    \end{equation*}
    The relevant constants depend only on $s,p$ and $q$ and the wavelet.
    Equivalently (via Lemma \ref{wavelet_bernstein}),
    \begin{equation*}
        \|f\|_{\dot{B}^s_{p,q}} \approx \left(\sum_{j\in \mathbb{Z}} 2^{jq(s+\frac{1}{2}-\frac{1}{p})}\left(\sum_{k\in \mathbb{Z}} |\langle f,\psi_{j,k}\rangle|^p\right)^{q/p}\right)^{1/q}.
    \end{equation*}
    The usual modifications are made if $p$ or $q$ is infinite.
\end{theorem}
Note that it is \emph{not} necessarily the case that $f \in \dot{B}^s_{p,q}(\mathbb{R})$ is equal to the sum of its wavelet series. That is, for a general locally integrable function $f$ it may not hold that
\begin{equation*}
    f = \sum_{j\in \mathbb{Z}} f_j
\end{equation*}
in any sense. For example, if $f$ is a polynomial of sufficiently small order then the above right hand side is zero \cite[Chapter 3, Proposition 4]{Meyer-wavelets-1992}. This issue is parallel
to the representation of $f$ by a Littlewood-Paley decomposition discussed in Section \ref{besov_section}.

In the next lemma, we explain how Lipschitz functions belonging to $\dot{B}^{\frac{1}{p}}_{p^\sharp,p}(\mathbb{R})$, can be expressed as a limit of wavelet series, up to a polynomial correction. We shall use the fact that if $f$
is a locally integrable function such that for all $j,k\in \mathbb{Z}$ we have
\begin{equation*}
    \langle f,\psi_{j,k}\rangle = 0
\end{equation*}
then $f$ is a polynomial. This follows from the realisation of distributions modulo polynomials by wavelet series, as in \cite[Section 6, Theorem 4(ii)]{Bourdaud-ondelette-1995}.
\begin{lemma}\label{besov_realisation}
    Let $f$ be a Lipschitz function on $\mathbb{R}$ such that $f \in \dot{B}^{\frac{1}{p}}_{p^\sharp,p}(\mathbb{R})$, where $0 < p \leq 1$. There exists a constant $c\in \mathbb{R}$ such that
    \begin{equation*}
        f(t) = f(0)+ct+\sum_{j\in \mathbb{Z}} (f_j(t)-f_j(0)),\quad t \in \mathbb{R}
    \end{equation*}
    and the series $\sum_{j\in \mathbb{Z}} f_j(t)-f_j(0)$ converges uniformly on compact sets. Moreover, $c$ can be chosen such that
    \begin{equation*}
        |c| \lesssim \|f'\|_\infty+\|f\|_{\dot{B}^{\frac{1}{p}}_{p^\sharp,p}}
    \end{equation*}
\end{lemma}
\begin{proof}
    For all $j\in \mathbb{Z}$ we have the Bernstein-type inequality \cite[Chapter 2, Theorem 3]{Meyer-wavelets-1992}
    \begin{equation*}
        \|f_j'\|_\infty \lesssim 2^j\|f_j\|_\infty.
    \end{equation*}
    It follows from \eqref{sequential_bernstein} that
    \begin{equation*}
        \|f_j\|_{\infty} \lesssim 2^{j\left(\frac{1}{p}-1\right)}\|f_j\|_{p^\sharp}
    \end{equation*}
    and since $p \leq 1$,
    \begin{equation*}
        \sum_{j\in \mathbb{Z}} \|f_j'\|_\infty \lesssim \|f\|_{\dot{B}^{\frac{1}{p}}_{p^\sharp,1}}\lesssim \|f\|_{\dot{B}^{\frac{1}{p}}_{p^\sharp,p}}.
    \end{equation*}
    Since the wavelet $\psi$ has been assumed to be $C^r$ for some $r>1$, for every $j\in\mathbb{Z}$ the function $f_j'$ is continuous.
    Hence the series $\sum_{j\in \mathbb{Z}} f_j'$ converges to a continuous function on $\mathbb{R}$. It follows that
    \begin{equation*}
        f'-\sum_{j\in \mathbb{Z}} f_j'
    \end{equation*}
    is a well-defined element of $L_\infty(\mathbb{R})$. Since the series converges uniformly, the function defined by
    \begin{equation*}
        g(t) := f(t)-f(0)-\sum_{j\in \mathbb{Z}} (f_j(t)-f_j(0)),\quad t\in \mathbb{R}
    \end{equation*}
    converges uniformly on compact subsets of $\mathbb{R}$ and due to having continuous derivative is absolutely continuous. By the triangle inequality, we have $\|g'\|_\infty\leq \|f'\|_\infty+\sum_{j\in \mathbb{Z}} \|f_j'\|_\infty\lesssim \|f'\|_\infty+\|f\|_{\dot{B}^{\frac{1}{p}}_{p^\sharp,p}}$. Since the series $\sum_{j\in \mathbb{Z}} f_j(t)-f_j(0)$ converges uniformly on compact subsets and $\psi$ is compactly supported, it follows that
    \begin{equation*}
        \langle g,\psi_{j,k}\rangle = 0,\quad j,k\in \mathbb{Z}.
    \end{equation*}
    The vanishing of all wavelet coefficients implies that $g$ is a polynomial (see the discussion preceding the theorem). Since $g'$ is a bounded polynomial, we must have that $g'$ is constant and hence there exists $c \in \mathbb{C}$ such that
    \begin{equation*}
        f(t)=f(0)+ct+\sum_{j\in \mathbb{Z}} f_j(t)-f_j(0),\quad t \in \mathbb{R}.
    \end{equation*}
    By our construction we have $|c| = \|g'\|_\infty \lesssim \|f'\|_\infty+\|f\|_{\dot{B}^\frac{1}{p}_{p^\sharp,p}}$.
\end{proof}

\subsection{Peller's sufficient condition revisited}
Peller's criterion \cite{Peller-besov-1990} is that if $f$ is a Lipschitz function belonging to $\dot{B}^{1}_{\infty,1}(\mathbb{R})$ then $f$ is operator Lipschitz (equivalently, $\mathcal{L}_1$-operator Lipschitz). In this subsection we explain how the decomposition of $f$ into a wavelet series leads to a new proof of this result. The ideas developed in this proof will be later used in the proof of Theorem \ref{main_sufficient_condition}, which is a more general assertion.

Note that we have the following homogeneity property: if $f_{(\lambda)}$ denotes the function $f_{(\lambda)}(t) = f(\lambda t)$, then
\begin{equation}\label{m_1_homogeneity}
    \|f_{(\lambda)}^{[1]}\|_{\mathrm{m}_1} = \lambda\|f^{[1]}\|_{\mathrm{m}_1}.
\end{equation}

Peller's original proof of the sufficiency of $\dot{B}^1_{\infty,1}$ is based on the following estimate \cite{Peller-besov-1990}:
if $f \in L_\infty(\mathbb{R})$ has Fourier transform supported in the interval $[-\sigma,\sigma]$. Then
\begin{equation}\label{peller_fav_thm}
    \|f^{[1]}\|_{\mathrm{m}_1} \lesssim \sigma\|f\|_\infty.
\end{equation}

Our proof differs from the original proof of Peller, and in place of \eqref{peller_fav_thm} we prove that for all locally integrable functions $f$ on $\mathbb{R}$ we have
$$
    \|f_j^{[1]}\|_{\mathrm{m}_1} \lesssim 2^{j}\|f_j\|_{\infty},\quad j\in \mathbb{Z}
$$
where $f_j$ is computed relative to a compactly supported $C^3$ wavelet.
This will follow as a consequence of the following result:
\begin{theorem}\label{elementary_estimate}
    Let $\phi \in C_c^3(\mathbb{R})$ be a compactly supported $C^3$ function, let $\alpha = \{\alpha_k\}_{k\in \mathbb{Z}}$ be a bounded sequence of complex numbers and let $\lambda > 0$. Define
    \begin{equation*}
        \phi_{\alpha,\lambda}(t) := \sum_{k\in \mathbb{Z}} \alpha_k\phi(\lambda t-k)
    \end{equation*}
    Then
    \begin{equation*}
        \|\phi_{\alpha,\lambda}^{[1]}\|_{\mathrm{m}_1} \lesssim \lambda\sup_{k \in \mathbb{Z}} |\alpha_k|.
    \end{equation*}
    The implied constant depends on $\phi$, but not on $\alpha$ or $\lambda$.
\end{theorem}    

In preparation for the proof of Theorem \ref{elementary_estimate}, we record
some useful facts about Schur multipliers of $\mathcal{L}_1$.
\begin{proposition}\label{schur_properties}
    Let $\phi:\mathbb{R}^2\to \mathbb{C}$ be a bounded function.
    \begin{enumerate}[{\rm (i)}]
        \item\label{single_variable} If $\phi$ depends only on one variable then $\|\phi\|_{\mathrm{m}_1}\leq\|\phi\|_\infty$.
        \item\label{toeplitz_form} Suppose that $\phi$ has Toeplitz form. That is, there exists a bounded function $\eta$ such that $\phi(t,s) = \eta(t-s)$. Then
        \begin{equation*}
            \|\phi\|_{\mathrm{m}_1} \leq (2\pi)^{-1}\|\widehat{\eta}\|_1
        \end{equation*}
        where $\widehat{\eta}(\xi) = \int_{-\infty}^\infty e^{-it\xi}\eta(t)\,dt$ is the Fourier transform of $\eta$.
        \item\label{direct_sum} Suppose that $\phi = \sum_{n\in \mathbb{Z}} \phi_n$, where the functions $\phi_n$ are have disjoint supports in both variables. That is, if $\phi_n(t,s)\neq 0$, then for all $m\neq n$ and $r \in \mathbb{R}$ we have that $\phi_m(t,r)=0$ and $\phi_m(r,s)=0$. Then
        \begin{equation*}
            \|\phi\|_{\mathrm{m}_1} = \sup_{n\in \mathbb{Z}} \|\phi_n\|_{\mathrm{m}_1}.
        \end{equation*}
        \item\label{submultiplicative} If $\chi$ is a second bounded function on $\mathbb{R}^2$, then
        \begin{equation*}
            \|\chi\phi\|_{\mathrm{m}_1} \leq \|\chi\|_{\mathrm{m}_1}\|\phi\|_{\mathrm{m}_1}.
        \end{equation*}
        Compare \eqref{algebra}.
        \item\label{C2_sufficiency} If $\phi$ is a compactly supported $C^3$ function, then $\phi^{[1]}$ is a bounded $\mathcal{L}_1$-Schur multiplier.
    \end{enumerate}
\end{proposition}
The nontrivial only components of the above proposition not already covered in the preliminaries are \eqref{toeplitz_form} and \eqref{C2_sufficiency}. 
To prove \eqref{toeplitz_form}, it only suffices to represent $\phi(t,s)$ as
\begin{equation*}
    \phi(t,s) = (2\pi)^{-1} \int_{-\infty}^\infty e^{i\xi t}e^{-i\xi s}\widehat{\eta}(\xi)\,d\xi
\end{equation*} 
and note that
\begin{equation*}
    \|\phi\|_{\mathrm{m}_1} \leq (2\pi)^{-1}\int_{-\infty}^\infty \|\{e^{i\xi t}\}_{t,s\in \mathbb{R}}\|_{\mathrm{m}_1}\|\{e^{-i\xi s}\}_{t,s\in \mathbb{R}}\|_{\mathrm{m}_1}|\widehat{\eta}(\xi)|\,d\xi = (2\pi)^{-1}\|\widehat{\eta}\|_1.
\end{equation*}
The assertion \eqref{C2_sufficiency} is that $C^3_c(\mathbb{R})$ functions $\phi$ are operator Lipschitz. This follows from the fact that the Fourier transform of
the derivative of $\phi$ is integrable. Alternatively, see Theorem \ref{Cbeta_sufficiency} below.

\begin{lemma}\label{wiener_class_lemma}
    Let $\rho$ be a compactly supported smooth function on $\mathbb{R}$ equal to $1$ in a neighbourhood of zero. The Fourier transform of the $L_2$-function
    \begin{equation*}
        \eta(t) = \frac{1-\rho(t)}{t}, t\in \mathbb{R}
    \end{equation*}
    is integrable.
\end{lemma} 
\begin{proof}
    As a non-absolutely convergent integral, we have
    \begin{equation*}
        \widehat{\eta}(\xi) = \int_{-\infty}^\infty e^{-i\xi t}\frac{1-\rho(t)}{t}\,dt.
    \end{equation*}
    Applying integration by parts $k$ times yields
    \begin{equation*}
        (i\xi)^k\widehat{\eta}(\xi) = \int_{-\infty}^\infty e^{-i\xi t}\left(\frac{d}{dt}\right)^k\left(\frac{1-\rho(t)}{t}\right)\,dt.
    \end{equation*}
    When $k > 1$, this defines an absolutely convergent integral and it follows
    that $\xi^k\widehat{\eta}(\xi)$ is uniformly bounded in $\xi$ for all $k>1$.
    Hence, $\widehat{\eta}(\xi)$ has rapid decay as $\xi\to \pm\infty$.

    Since $\eta$ belongs to $L_2(\mathbb{R}),$ the Fourier transform $\widehat{\eta}$ also belongs to $L_2(\mathbb{R}).$ In particular, $\widehat{\eta}$ is locally integrable.
    Hence $\widehat{\eta}$ has rapid decay at infinity and is integrable near zero. Thus $\widehat{\eta}$ is integrable over $\mathbb{R}.$
\end{proof}
    
\begin{proof}[Proof of Theorem \ref{elementary_estimate}]
    By the homogeneity property \eqref{m_1_homogeneity}, it suffices to take $\lambda =1$, and for brevity we denote $\phi_{\alpha} = \phi_{\alpha,1}$.

    Let $\rho$ be a smooth compactly supported function on $\mathbb{R}$ such that $\rho$ is identically $1$ in a neighbourhood of zero, define $\eta(t) = \frac{1-\rho(t)}{t}$
    as in Lemma \ref{wiener_class_lemma}.
    
    We split the divided difference of $\phi_{\alpha}$ as
    \begin{equation}\label{diagonal_splitting}
        \phi_{\alpha}^{[1]}(t,s) = \phi_{\alpha}^{[1]}(t,s)\rho(t-s) + \phi_{\alpha}^{[1]}(t,s)(1-\rho(t-s)) \stackrel{\mathrm{def}}{=} A(t,s)+B(t,s). 
    \end{equation}
    We bound each summand separately. 
    
    For the second summand in \eqref{diagonal_splitting}, we have by definition that
    \begin{align*}
        B(t,s) &:= \sum_{k\in \mathbb{Z}} \alpha_k \frac{\phi(t-k)-\phi(s-k)}{t-s}(1-\rho(t-s))\\
        &= \left(\sum_{k\in \mathbb{Z}} \alpha_k(\phi(t-k)-\phi(s-k))\right)\eta(t-s)\\
        &= \sum_{k\in \mathbb{Z}} \alpha_k\phi(t-k)\eta(t-s) - \sum_{k\in \mathbb{Z}} \alpha_k\phi(s-k)\eta(t-s).
    \end{align*}
    Using Proposition \ref{schur_properties}.\eqref{submultiplicative} and the triangle inequality, we have
    \begin{equation*}
        \|B\|_{\mathrm{m}_1} \leq \|\sum_{k\in \mathbb{Z}} \alpha_k\phi(t-k)\|_{\mathrm{m}_1}\|\eta(t-s)\|_{\mathrm{m}_1} + \|\sum_{k\in \mathbb{Z}} \alpha_k\phi(s-k)\|_{\mathrm{m}_1}\|\eta(t-s)\|_{\mathrm{m}_1}.
    \end{equation*}
    By Lemma \ref{wiener_class_lemma}, the Fourier transform of $\eta$ is integrable
    and hence Proposition \ref{schur_properties}.\eqref{toeplitz_form} implies 
    that $\|\eta(t-s)\|_{\mathrm{m}_1} < \infty$. Proposition \ref{schur_properties}.\eqref{single_variable} implies that
    \begin{equation*}
        \|\sum_{k\in \mathbb{Z}} \alpha_k\phi(\cdot -k)\|_{\mathrm{m}_1}\lesssim \sup_{k\in \mathbb{Z}} |\alpha_k|.
    \end{equation*}
    Thus,
    \begin{equation*}
        \|B\|_{\mathrm{m}_1} \lesssim \sup_{k\in \mathbb{Z}} |\alpha_k|.
    \end{equation*}
    
    Now we bound the first summand in \eqref{diagonal_splitting}. We may assume that $\rho$ is supported in the interval $(-1,1)$. It follows that
    the function $(t,s) \mapsto \phi_{\alpha}^{[1]}(t,s)\rho(t-s)$ is supported in the strip
    $\{(t,s)\in \mathbb{R}^2\;:\; |t-s|< 1\}.$
    Observe that we have
    \begin{equation*}
        \{(t,s)\in \mathbb{R}^2\;:\; |t-s|< 1\} \subset \bigcup_{n\in \mathbb{Z},|j|<2} [n,n+1)\times [n+j,n+j+1).
    \end{equation*}
    Therefore,
    \begin{align*}
        A(t,s) := \phi_{\alpha}^{[1]}(t,s)\rho(t-s) &= \sum_{|j|<2}\sum_{n\in \mathbb{Z}} \phi_{\alpha}^{[1]}(t,s)\rho(t-s)\chi_{[n,n+1)}(t)\chi_{[n+j,n+j+1)}(s)\\
                                            &= \sum_{|j|<2}\left(\sum_{n,k\in \mathbb{Z}} \alpha_k\frac{\phi(t-k)-\phi(s-k)}{t-s}\rho(t-s)\chi_{[n,n+1)}(t)\chi_{[n+j,n+j+1)}(s)\right)\\
                                            &= \sum_{|j|<2} F_j(t,s)
    \end{align*}
    where for each $|j|<2$ we have denoted
    \begin{equation*}
        F_j(t,s) := \sum_{n,k\in \mathbb{Z}} \alpha_k\frac{\phi(t-k)-\phi(s-k)}{t-s}\rho(t-s)\chi_{[n,n+1)}(t)\chi_{[n+j,n+j+1)}(s).
    \end{equation*}
    This function can be written in the form 
    \begin{equation*}
        F_j(t,s) = \sum_{n\in \mathbb{Z}} \chi_{[n,n+1)}(t)\chi_{[n+j,n+j+1)}(s)\sum_{k\in \mathbb{Z}} \alpha_k \frac{\phi(t-k)-\phi(s-k)}{t-s}\rho(t-s)\chi_{[n,n+1)}(t)\chi_{[n+j,n+j+1)}(s). 
    \end{equation*}
    Hence, $F_j$ has the form required in Proposition \ref{schur_properties}.\eqref{direct_sum}. Thus,
    \begin{equation*}
        \|F_j\|_{\mathrm{m}_1} = \sup_{n\in \mathbb{Z}} \|\sum_{k\in \mathbb{Z}} \alpha_k\frac{\phi(t-k)-\phi(s-k)}{t-s}\rho(t-s)\chi_{[n,n+1)}(t)\chi_{[n+j,n+j+1)}(s)\|_{\mathrm{m}_1}.
    \end{equation*}
    Since $\phi$ is compactly supported, for each $n$ the sum over $k$ has only finitely many terms. In fact, there exists a constant $N$ (depending on $\phi$ and $j$) such that for $|n-k|> N$ we have
    \begin{equation*}
        \frac{\phi(t-k)-\phi(s-k)}{t-s}\rho(t-s)\chi_{[n,n+1)}(t)\chi_{[n+j,n+j+1)}(s) = 0
    \end{equation*}
    Therefore,
    \begin{equation*}
        \|F_j\|_{\mathrm{m}_1} = \sup_{n\in \mathbb{Z}} \|\sum_{|k-n|\leq N}\alpha_k \frac{\phi(t-k)-\phi(s-k)}{t-s}\rho(t-s)\chi_{[n,n+1)}(t)\chi_{[n+j,n+j+1)}(s)\|_{\mathrm{m}_1}.
    \end{equation*}
    Since $\phi\in C^3$, the divided difference $(t,s)\mapsto \frac{\phi(t-k)-\phi(s-k)}{t-s}$
    is a bounded Schur multiplier with norm independent of $k$ by Proposition \ref{schur_properties}.\eqref{C2_sufficiency}. Similarly, since the Fourier transform of $\rho$ is Schwartz class the function $(t,s) \mapsto \rho(t-s)$ is a bounded Schur multiplier by Proposition \ref{schur_properties}.\eqref{toeplitz_form}. It follows that
    \begin{equation*}
        \|F_j\|_{\mathrm{m}_1} \lesssim \sup_{n\in \mathbb{Z}} \sum_{|k-n|<N} |\alpha_k| \lesssim_N \sup_{k\in \mathbb{Z}} |\alpha_k|.
    \end{equation*}
    By the triangle inequality, it follows that
    \begin{equation*}
        \|A\|_{\mathrm{m}_1} \lesssim \sup_{k\in \mathbb{Z}} |\alpha_k|.
    \end{equation*}
    Finally, from \eqref{diagonal_splitting} we have
    \begin{equation*}
        \|\phi^{[1]}_{\alpha}\|_{\mathrm{m}_1} \lesssim \sup_{k\in \mathbb{Z}} |\alpha_k|.
    \end{equation*}
\end{proof}
Using Lemma \ref{wavelet_bernstein}, we can deduce the following substitute for \eqref{peller_fav_thm}.
Recall that $f_j = \sum_{k\in \mathbb{Z}} \psi_{j,k}\langle f,\psi_{j,k}\rangle.$
\begin{lemma}\label{peller_fav_easy}
    Let $f$ be a locally integrable function on $\mathbb{R}$, and let $j\in \mathbb{Z}$ be such that $f_j$ is bounded where $f_j$ is computed with respect to a compactly supported $C^3$ wavelet $\psi$. We have
    \begin{equation*}
        \|f_j^{[1]}\|_{\mathrm{m}_1} \lesssim 2^{j}\|f_j\|_\infty.
    \end{equation*}
\end{lemma}
\begin{proof}
    This is essentially a special case of Theorem \ref{elementary_estimate}. We have
    \begin{equation*}
        f_j(t) = \sum_{k\in \mathbb{Z}} 2^{\frac{j}{2}} \psi(2^j t-k)\langle f,\psi_{j,k}\rangle,\quad t \in \mathbb{R}.
    \end{equation*}
    Theorem \ref{elementary_estimate} and \eqref{disjoint_supports} together yield
    \begin{equation*}
        \|f_j^{[1]}\|_{\mathrm{m}_1} \leq 2^{j}\sup_{k\in \mathbb{Z}} 2^{\frac{j}{2}}|\langle f,\psi_{j,k}\rangle| \approx 2^j\|f_j\|_{\infty}.
    \end{equation*}
\end{proof}

Finally, we achieve Peller's sufficient condition.
\begin{corollary}\label{Peller_critereon}
    Let $f \in \dot{B}^1_{\infty,1}(\mathbb{R})$ be Lipschitz. Then $\|f^{[1]}\|_{\mathrm{m}_1}< \infty$, and
    \begin{equation*}
        \|f^{[1]}\|_{\mathrm{m}_1} \lesssim \|f'\|_{\infty} + \|f\|_{\dot{B}^1_{\infty,1}}.
    \end{equation*}
\end{corollary}
\begin{proof}
    We apply the representation of $f$ from Lemma \ref{besov_realisation}. We have
    \begin{equation*}
        f(t) = f(0)+ct+\sum_{j\in \mathbb{Z}} f_j(t)-f_j(0),\quad t \in \mathbb{R}
    \end{equation*}
    where $f_j$ is computed relative to a compactly supported $C^3$-wavelet.
    By the triangle inequality, the bound in Lemma \ref{besov_realisation} and Lemma \ref{peller_fav_easy}, we have
    $$
        \|f^{[1]}\|_{\mathrm{m}_1} \leq |c|+\sum_{j\in \mathbb{Z}} \|f_j^{[1]}\|_{\mathrm{m}_1} \lesssim \|f'\|_\infty+\|f\|_{\dot{B}^1_{\infty,1}}+\sum_{j\in \mathbb{Z}} 2^j\|f_j\|_\infty = \|f'\|_\infty+\|f\|_{\dot{B}^1_{\infty,1}}.
    $$
\end{proof}

\subsection{Sufficient conditions for a function to be $\mathcal{L}_p$-operator Lipschitz}\label{sufficency_section}
We now present sufficient conditions for a function to be $\mathcal{L}_p$-operator Lipschitz for $0 < p < 1$. The $p=1$ case has been covered by Corollary \ref{Peller_critereon}, however
the arguments in this subsection apply for $p=1$.
    
Recall that $\|\cdot\|_{\mathrm{m}_p}$ denotes the $\mathcal{L}_p$-bounded Schur multiplier norm. Our proofs are based on the following result.
\begin{theorem}\label{Cbeta_sufficiency}
    Let $0 < p \leq 1$. If $f \in C^\beta_c(\mathbb{R})$ is compactly supported, where $\beta > \frac{2}{p}$, then $\|f^{[1]}\|_{\mathrm{m}_p} < \infty$. 
\end{theorem}
The result should be compared with Theorem 9.6 of \cite{Birman-Solomyak-survey-1977}, which is very similar. The proof we give here is based on first
proving that if $\phi \in C^\beta(\mathbb{T})$, then $\{\phi^{[1]}(z,w)\}_{z,w \in \mathbb{T}}$ is an $\mathcal{L}_p$-bounded Schur multiplier. This is actually an immediate consequence
of \eqref{circle_lipschitz}, since $C^{\beta}(\mathbb{T}) \subset B^{\frac{1}{p}}_{\infty,p}(\mathbb{T})$ when $\beta > \frac{1}{p}$. 
If $\phi$ is a function of class $C^{\beta}(\mathbb{T})$, supported in a compact subset of $\mathbb{T}\setminus \{1\}$, then the Cayley transform 
sends $\phi$ to a compactly supported function of class $C^\beta$ on $\mathbb{R}$. We use the much stronger assumption that $\beta> \frac{2}{p}$ because
it suffices for our purposes and we can give an especially elementary proof. Ultimately, the same result with only $\beta>\frac{1}{p}$ follows from Theorem \ref{main_sufficient_condition}.

\begin{proof}[Proof of Theorem \ref{Cbeta_sufficiency}]
    Initially we prove the corresponding result on the circle $\mathbb{T}$. Note that for $n\in \mathbb{Z}$ we have
    \begin{equation*}
        \frac{z^n-w^n}{z-w} = \sum_{k=0}^{n-1} z^kw^{n-k-1},\quad z,w\in \mathbb{T}
    \end{equation*}
    and therefore the $p$-triangle inequality for the $\mathrm{m}_p$-quasi-norm implies
    \begin{equation*}
        \left\|\frac{z^n-w^n}{z-w}\right\|_{\mathrm{m}_p} \lesssim n^{\frac{1}{p}}
    \end{equation*}
    with a constant independent of $n$. If $h \in C^{\beta}(\mathbb{T})$, then the Fourier coefficients $\{\widehat{h}(n)\}_{n\in \mathbb{Z}}$ obey
    \begin{equation*}
        |\widehat{h}(n)|\lesssim (1+|n|)^{-\beta}.
    \end{equation*}
    See e.g. \cite[Theorem 3.2.9(b)]{Grafakos-1}.
    For all $z\neq w\in \mathbb{T}$ we have
    \begin{equation*}
        h^{[1]}(z,w) := \frac{h(z)-h(w)}{z-w} = \sum_{n\in \mathbb{Z}} \widehat{h}(n)\frac{z^n-w^n}{z-w}.
    \end{equation*}
    Using the $p$-triangle inequality for the $\mathrm{m}_p$-norm \eqref{p-prop}, it follows that
    \begin{equation*}
        \|h^{[1]}\|_{\mathrm{m}_p}^p \leq \sum_{n\in \mathbb{Z}} |\widehat{h}(n)|^{ p}\left\|\{\frac{z^n-w^n}{z-w}\}_{z,w\in \mathbb{T}}\right\|^{ p} \lesssim \sum_{n\in \mathbb{Z}} |n|(1+|n|)^{-\beta p}
    \end{equation*}
    Since $\beta > \frac{2}{p}$, this series converges and hence $h^{[1]}$ is a Schur multiplier of $\mathcal{L}_p$.
    
    Now let $f \in C^\beta_c(\mathbb{R})$, and consider the image under the Cayley transform,
    \begin{equation*}
        h(z) := f\left(i\frac{z+1}{z-1}\right).
    \end{equation*} 
    Since $f \in C^\beta_c(\mathbb{R})$ and the Cayley transform is Lipschitz on compact subsets of $\mathbb{R}$, it follows that $h \in C^\beta(\mathbb{T})$.
    Therefore $\|h^{[1]}\|_{\mathrm{m}_p} < \infty$, and that $h$ is $\mathcal{L}_p$-operator Lipschitz for differences of unitary operators in the sense
    \begin{equation*}
        \|h(U)-h(V)\|_p\lesssim \|U-V\|_p
    \end{equation*} 
    for all unitaries $U$ and $V$ such that $U-V \in \mathcal{L}_p$. If $A$ and $B$ are self-adjoint operators such that $A-B \in \mathcal{L}_p$, we define
    \begin{equation*}
        U = \frac{A+i}{A-i},\quad V = \frac{B+i}{B-i}.
    \end{equation*} 
    Then
    \begin{equation*}
        U-V = \frac{2i}{A-i}-\frac{2i}{B-i} = 2i(B-i)^{-1}(B-A)(A-i)^{-1} \in \mathcal{L}_p.
    \end{equation*}
    We also have $f(A) = h(U)$ and $f(B) = h(V)$. Therefore,
    \begin{equation*}
        \|f(A)-f(B)\|_p = \|h(U)-h(V)\|_p\lesssim \|U-V\|_p \lesssim \|(B-i)^{-1}\|_\infty\|(A-i)^{-1}\|_\infty\|A-B\|_p\lesssim \|A-B\|_p.
    \end{equation*} 
    It follows from Theorem \ref{lp_lip_implies_lp_schur} that $f^{[1]}$ is a Schur multiplier of $\mathcal{L}_p$.
\end{proof}
Note that with $p < 1$ we still have the following homogeneity result, identical to \eqref{m_1_homogeneity}:
\begin{equation}\label{schur_homogeneity}
    \|f^{[1]}_{(\lambda)}\|_{\mathrm{m}_p} = \lambda\|f^{[1]}\|_{\mathrm{m}_p}
\end{equation}
where $f_{(\lambda)}(t) := f(\lambda t)$.
The most important component of our proof of Theorem \ref{main_sufficient_condition} is as follows. Recall that we denote
\begin{equation*}
    p^\sharp = \frac{p}{1-p}.
\end{equation*}
\begin{theorem}\label{quasi_wavelet_bernstein}
    Let $\phi$, $\alpha$, $\lambda$ be as in Theorem \ref{elementary_estimate}, but now assume that $\phi \in C^\beta_c(\mathbb{R})$ where $\beta > \frac{2}{p}$. Let $0 < p \leq 1$. Then
    \begin{equation*}
        \|\phi_{\alpha,\lambda}^{[1]}\|_{\mathrm{m}_p} \lesssim \lambda \|\alpha\|_{\ell_{p^\sharp}}.
    \end{equation*}
\end{theorem}
Theorem \ref{quasi_wavelet_bernstein} generalises Theorem \ref{elementary_estimate}, and our proof is similar. We first record some useful properties of Schur multipliers
of $\mathcal{L}_p$. 

\begin{lemma}\label{push_to_floor}
    Let $\phi:\mathbb{Z}^2\to \mathbb{C}$. Then
    \begin{equation*}
        \|\{\phi(\lfloor t\rfloor,\lfloor s\rfloor)\}_{t,s\in \mathbb{R}}\|_{\mathrm{m}_p} = \|\{\phi(j,k)\}_{j,k\in\mathbb{Z}}\|_{\mathrm{m}_p}.
    \end{equation*}
\end{lemma}
\begin{proof}
    Let $\{s_j\}_{j=1}^K$ and $\{t_k\}_{k=1}^K$ be finite subsets of $\mathbb{R}$. Assume that 
    \begin{equation*}
        \max_{1\leq j,k\leq K} \{|s_j|,|t_k|\} \leq N\in \mathbb{N}. 
    \end{equation*}
    By enlarging $K$ and adding additional points to the sequences $\{s_j\}_{j=1}^K$ and $\{t_k\}_{k=1}^K$ if necessary, we assume that there exists $n > 0$ such that for all $-N\leq l< N$ we have
    \begin{equation*}
        n = |\{s_j\}_{j=1}^K\cap [l,l+1)| = |\{t_k\}_{k=1}^K\cap [l,l+1)|.
    \end{equation*}
    We now relabel the sequences as $s_{j,l}$ and $t_{k,m}$, where $1\leq j,k\leq n$ and $-N\leq l,m < N$ such that
    \begin{equation*}
        s_{j,l},\, t_{k,l}\in [l,l+1),\quad -N\leq l<N.
    \end{equation*}
    Denote by $\mathrm{id}_{M_n(\mathbb{C})}$ the $n\times n$ matrix of ones. We have
    \begin{equation*}
        \{\phi(\lfloor s_{j,l}\rfloor ,\lfloor t_{k,m}\rfloor)\}_{1\leq j,k\leq n,\,-N\leq l,m< N} = \{\phi(l,m)\}_{-N\leq l,m< N} \otimes \mathrm{id}_{M_n(\mathbb{C})}.
    \end{equation*}
    Due to the automatic complete boundedness property (Theorem \ref{acb}), it follows that
    \begin{equation*}
        \|\{\phi(\lfloor s_{j,l}\rfloor ,\lfloor t_{k,m}\rfloor)\}_{1\leq j,k\leq n,\,-N\leq l,m< N}\|_{\mathrm{m}_p} = \|\{\phi(l,m)\}_{-N\leq l,m< N}\|_{\mathrm{m}_p} \leq \|\{\phi(l,m)\}_{l,m\in \mathbb{Z}}\|_{\mathrm{m}_p}.
    \end{equation*}
    Since adding rows and columns to a matrix can only increase the $\mathrm{m}_p$-norm, it follows that for arbitrary sequences $\{s_{j}\}_{j=1}^K$ and $\{t_k\}_{k=1}^K$ we have
    \begin{equation*}
        \|\{\phi(\lfloor s_j\rfloor,\lfloor t_k\rfloor)\}_{1\leq j,k\leq K}\|_{\mathrm{m}_p} \leq \|\{\phi(l,m)\}_{l,m\in \mathbb{Z}}\|_{\mathrm{m}_p}.
    \end{equation*}
    Taking the supremum over all sequences yields
    \begin{equation*}
        \|\{\phi(\lfloor s\rfloor, \lfloor t\rfloor)\}_{t,s\in \mathbb{R}}\|_{\mathrm{m}_p} \leq \|\{\phi(l,m)\}_{l,m\in \mathbb{Z}}\|_{\mathrm{m}_p}.
    \end{equation*}
    The reverse inequality is trivial.
\end{proof}

One further property we need is that if $\lambda = \{\lambda_j\}_{j\in \mathbb{Z}}$ is a scalar sequence and $A = \{A_{j,k}\}_{j,k\in \mathbb{Z}}$ is a matrix then
\begin{equation}\label{left_multiplication}
    \|\{\lambda_jA_{j,k}\}_{j,k\in \mathbb{Z}}\|_{\mathrm{m}_p} \leq \|\{\lambda_j\}_{j\in \mathbb{Z}}\|_{\ell_{p^\sharp}}\|A\|_{\mathrm{m}_1}. 
\end{equation}
Indeed, if $\Lambda$ denotes the diagonal matrix with entries $\{\lambda_j\}_{j\in \mathbb{Z}}$, then by H\"older's inequality \eqref{operator_holder_inequality} and Lemma \ref{rank_one_suffices} we have
\begin{align*}
    \|\{\lambda_jA_{j,k}\}_{j,k\in \mathbb{Z}}\|_{\mathrm{m}_p} &= \sup_{\|\xi\|,\|\eta\|\leq 1} \|\{\lambda_jA_{j,k}\xi_j\eta_k\}_{j,k\in \mathbb{Z}}\|_p\\
                                                    &= \sup_{\|\xi\|,\|\eta\|\leq 1} \|\Lambda (A\circ (\xi\otimes \eta))\|_{p}\\
                                                    &\leq \|\Lambda\|_{p^\sharp}\sup_{\|\xi\|,\|\eta\|\leq 1} \|A\circ (\xi\otimes \eta)\|_1\\
                                                    &= \|\Lambda\|_{p^\sharp}\|A\|_{\mathrm{m}_1}.
\end{align*}
Since $\|\Lambda\|_{p^\sharp} = \|\{\lambda_j\}_{j\in \mathbb{Z}}\|_{\ell_{p^\sharp}}$, this proves \eqref{left_multiplication}.

Our method of proof of Theorem \ref{quasi_wavelet_bernstein} is conceptually similar to that of Theorem \ref{elementary_estimate}, but in place of the function $(t,s)\mapsto 1-\rho(t-s)$
it is more convenient to use the following discretised version:
\begin{equation*}
    (t,s)\mapsto \chi_{|\lfloor t\rfloor -\lfloor s\rfloor|> R}
\end{equation*}
where $R>1$ is sufficiently large, depending on $p$.
\begin{lemma}\label{vanishing_toeplitz_lemma}
    Let $\alpha = \{\alpha_k\}_{k\in \mathbb{Z}}$ be a scalar sequence, and let $g_{\alpha}$ be the function
    \begin{equation*}
        g_{\alpha}(t) = \sum_{k\in \mathbb{Z}} \alpha_k\chi_{[k,k+1)}(t).
    \end{equation*}
    Then for all $n\geq 1$ and $R > 1$ we have
    \begin{equation*}
        \|\{g_{\alpha}(t)\frac{\chi_{|\lfloor t\rfloor -\lfloor s\rfloor| > R}}{(\lfloor t\rfloor -\lfloor s\rfloor)^n}\}_{t,s\in \mathbb{R}}\|_{\mathrm{m}_p} \leq \left(\frac{2}{R}\right)^{n/2}\|\alpha\|_{\ell_{p^\sharp}}.
    \end{equation*}
\end{lemma}
\begin{proof}
    Note that $g_{\alpha}(t) = \alpha_{\lfloor t\rfloor}$. Using Lemma \ref{push_to_floor}, it suffices to prove that
    \begin{equation*}
        \|\{\alpha_j\frac{\chi_{|j-k|> R}}{(j-k)^n}\}_{j,k\in \mathbb{Z}}\|_{\mathrm{m}_p} \leq \left(\frac{2}{R}\right)^{\frac{n}{2}}\|\alpha\|_{\ell_{p^\sharp}}.
    \end{equation*}
    In fact, via \eqref{left_multiplication}, and repeatedly using \eqref{algebra}, it only suffices to check that
    \begin{equation*}
        \|\{\frac{\chi_{|j-k|> R}}{j-k}\}_{j,k\in \mathbb{Z}}\|_{\mathrm{m}_1} \leq \left(\frac{2}{R}\right)^{\frac{1}{2}}.
    \end{equation*}
    This is a Toeplitz matrix. Hence, by Proposition \ref{schur_properties}.\eqref{toeplitz_form} we have
    \begin{equation*}
        \|\{\frac{\chi_{|j-k| > R}}{j-k}\}_{j,k\in \mathbb{Z}}\|_{\mathrm{m}_1} \leq \|f\|_{L_1[0,1]}
    \end{equation*}
    where $f$ is the function
    \begin{equation*}
        f(t) = \sum_{|j|>R} \frac{1}{j}e^{2\pi i t j}.
    \end{equation*}
    Bounding the $L_1$-norm of $f$ by the $L_2$-norm, it follows from Plancherel's identity that
    \begin{equation*}
        \|\{\frac{\chi_{|j-k|> R}}{j-k}\}_{j,k\in \mathbb{Z}}\|_{\mathrm{m}_1} \leq \left(\sum_{|j|> R} \frac{1}{j^2}\right)^{\frac{1}{2}} \leq \left(\frac{2}{R}\right)^{\frac{1}{2}}.
    \end{equation*}
    This completes the proof.        
\end{proof}

Now we set $R=2^{3+\frac2p}$, and defining $g_{\alpha}$ as in Lemma \ref{vanishing_toeplitz_lemma} we define
$$G_{\alpha}(t,s):=g_{\alpha}(t)\cdot \frac{\chi_{|\lfloor s\rfloor-\lfloor t\rfloor|> R}}{t-s},$$
$$H_{\alpha}(t,s):=g_{\alpha}(s)\cdot \frac{\chi_{|\lfloor s\rfloor-\lfloor t\rfloor|> R}}{t-s}.$$

\begin{lemma}\label{GH lemma} In the notation above, for every $\alpha\in \ell_{p^{\sharp}},$ we have
$$\|G_{\alpha}\|_{\mathrm{m}_p},\|H_{\alpha}\|_{\mathrm{m}_p}\leq c_p\|\alpha\|_{\ell_{p^{\sharp}}}$$
where $c_p$ depends only on $p$.
\end{lemma}
\begin{proof}  
Denote by $\{t\}$ and $\{s\}$ the fractional parts of $t,s\in \mathbb{R}$ respectively, so that 
\begin{equation*}
    \frac{1}{t-s} = \frac{1}{\lfloor t\rfloor+\{t\}-\lfloor s\rfloor-\{s\}} = \frac{1}{\lfloor t\rfloor-\lfloor s\rfloor}\cdot \frac{1}{1-\frac{\{s\}-\{t\}}{\lfloor t\rfloor-\lfloor s\rfloor}}.
\end{equation*}
Since $R>8$, if $\lfloor t\rfloor-\lfloor s\rfloor>R$ then $\left|\frac{\{s\}-\{t\}}{\lfloor t\rfloor -\lfloor s\rfloor}\right|<1$, and hence for all $t,s\in \mathbb{R}$ we have a convergent series
\begin{align*}
    \frac{\chi_{|\lfloor s\rfloor -\lfloor t\rfloor|>R}}{t-s} &= \frac{\chi_{|\lfloor t\rfloor -\lfloor s\rfloor|>R}}{\lfloor t\rfloor -\lfloor s\rfloor} \cdot \frac{1}{1-\frac{\{s\}-\{t\}}{\lfloor t\rfloor-\lfloor s\rfloor}}\\
                                                                &= \frac{\chi_{|\lfloor t\rfloor -\lfloor s\rfloor|>R}}{\lfloor t\rfloor -\lfloor s\rfloor}\sum_{k=0}^\infty \left(\frac{\{s\}-\{t\}}{\lfloor t\rfloor -\lfloor s\rfloor}\right)^k.
\end{align*}
It follows from the $p$-triangle inequality for the $\mathrm{m}_p$-norm \eqref{p-prop} that
\begin{equation*}
    \|G_{\alpha}\|_{\mathrm{m}_p}^p \leq \sum_{k=0}^\infty \|\{g_{\alpha}(t)\frac{\chi_{|\lfloor s\rfloor-\lfloor t\rfloor|> R}}{(\lfloor t\rfloor -\lfloor s\rfloor)^{k+1}}(\{s\}-\{t\})^k\}_{t,s\in \mathbb{R}}\|_{\mathrm{m}_p}^p.
\end{equation*}
Using the submultiplicativity property of the $\mathrm{m}_p$-norm \eqref{algebra}, it follows that
\begin{equation*}
    \|G_{\alpha}\|_{\mathrm{m}_p}^p \leq \sum_{k=0}^\infty \|\{g_{\alpha}(t)\frac{\chi_{|\lfloor s\rfloor-\lfloor t\rfloor|> R}}{(\lfloor t\rfloor -\lfloor s\rfloor)^{k+1}}\}_{t,s\in \mathbb{R}}\|_{\mathrm{m}_p}^p\|\{\{t\}-\{s\}\}_{t,s\in \mathbb{R}}\|_{\mathrm{m}_p}^{kp}.
\end{equation*}
Since $\{t\}$ and $\{s\}$ are bounded above by $1$, we have
\begin{equation*}
    \|\{\{t\}-\{s\}\}_{t,s\in \mathbb{R}}\|_{\mathrm{m}_p}^{p} \leq 2
\end{equation*}
and therefore
\begin{equation*}
    \|\{\{t\}-\{s\}\}_{t,s\in \mathbb{R}}\|_{\mathrm{m}_p}^{kp} \leq 2^{k}.
\end{equation*}
It follows that
\begin{equation*}
    \|G_{\alpha}\|_{\mathrm{m}_p}^p \leq \sum_{k=0}^\infty 2^k  \|\{g_{\alpha}(t)\frac{\chi_{|\lfloor s\rfloor-\lfloor t\rfloor|> R}}{(\lfloor t\rfloor -\lfloor s\rfloor)^{k+1}}\}_{t,s\in \mathbb{R}}\|_{\mathrm{m}_p}^p.
\end{equation*}
Applying Lemma \ref{vanishing_toeplitz_lemma} and using $R = 2^{3+\frac{2}{p}}$ we have
\begin{align*}
    \|G_{\alpha}\|_{\mathrm{m}_p}^p &\leq \|\alpha\|_{\ell_{p^\sharp}}^p\sum_{k=0}^\infty 2^k \left(\frac{2}{R}\right)^{\frac{kp}{2}}\\
                            &= \|\alpha\|_{\ell_{p^\sharp}}^p\sum_{k=0}^\infty 2^{-kp}\\
                            &=c_p\|\alpha\|_{\ell_{p^\sharp}}^p.
\end{align*}
This proves the first inequality. The second identity follows from taking the transpose of the first.

\end{proof}

\begin{proof}[Proof of Theorem \ref{quasi_wavelet_bernstein}] 
Using the homogeneity property \eqref{schur_homogeneity}, it suffices to take $\lambda=1$, and we abbreviate $\phi_{\alpha} = \phi_{\alpha,1}$.

Without loss of generality, we may assume that the functions $\{\phi(\cdot-k)\}_{k\in \mathbb{Z}}$ are disjointly supported. Indeed, otherwise we may select $N>1$ sufficiently
large such that $\{\phi(\cdot-Nk)\}_{k\in \mathbb{Z}}$ are disjointly supported, and write
\begin{equation*}
\phi_{\alpha} = \sum_{j=0}^{N-1}\phi_{\alpha^{(j)}}
\end{equation*}
where $\alpha^{(j)}$ is the sequence $\{\alpha_{j+Nk}\}_{k\in \mathbb{Z}}$. Then we may prove the assertion for each $\phi_{\alpha^{(j)}}$ separately. Moreover, since the assertion
is invariant under rescaling, without loss of generality we assume that $\phi$ is supported in $(0,1)$.

Fix $R=2^{3+\frac2p}.$  We split up $\phi_{\alpha}^{[1]}$ as
\begin{equation}\label{diagonal_splitting_2}
\phi_{\alpha}^{[1]}(t,s)=\phi_{\alpha}^{[1]}(t,s)\chi_{|\lfloor s\rfloor-\lfloor t\rfloor|\leq R}+\phi_{\alpha}^{[1]}(t,s)\chi_{|\lfloor s\rfloor-\lfloor t\rfloor|> R}\stackrel{\mathrm{def}}{=}A_R(t,s)+B_R(t,s),\quad t,s\in \mathbb{R}.
\end{equation}
We bound the individual terms separately.

For the first summand, we have
\begin{equation*}
A_R=\sum_{|j|\leq R}F_j,\text{ where } F_j(t,s):=\phi_{\alpha}^{[1]}(t,s)\chi_{\lfloor s\rfloor-\lfloor t\rfloor=j}.
\end{equation*}
Clearly,
$$\chi_{\lfloor s\rfloor-\lfloor t\rfloor=j}=\sum_{n\in\mathbb{Z}}\chi_{[n,n+1)}(t)\chi_{[n+j,n+j+1)}(s).$$
Thus,
$$F_j=\sum_{n\in\mathbb{Z}}F_{n,j},\text{ where } F_{n,j}(t,s):=\sum_{n\in\mathbb{Z}}\chi_{[n,n+1)}(t)\phi^{[1]}_{\alpha}(t,s)\chi_{[n+j,n+j+1)}(s).$$

We have (see \eqref{p-prop})
\begin{equation}\label{ar preliminary estimate}
\|A_R\|_{\mathrm{m}_p}^p\leq\sum_{|j|\leq R}\|F_j\|_{\mathrm{m}_p}^p.
\end{equation}
Each $F_j$ has a generalised block-diagonal structure in the sense of Lemma \ref{block_diagonal_formula}. It follows from that Lemma that
\begin{equation}\label{fj 22 estimate}
\|F_j\|_{\mathrm{m}_p}\leq \Big\|\Big\{\|F_{n,j}\|_{\mathrm{m}_p}\Big\}_{j\in \mathbb{Z}}\Big\|_{\ell_{p^{\sharp}}}.
\end{equation}

For $j\neq0,$ we write
$$F_{n,j}=\alpha_n G_{n,j}-\alpha_{n+j}H_{n,j},$$
where
\begin{align*}
G_{n,j}(t,s)&:=\phi^{[1]}(t-n,s-n)\chi_{[n+j,n+j+1)}(s),\\
H_{n,j}(t,s)&:=\phi^{[1]}(t-n-j,s-n-j)\chi_{[n,n+1)}(t)\chi_{[n+j,n+j+1)}(s).
\end{align*}
In particular,
\begin{equation*}
\|F_{n,j}\|_{\mathrm{m}_p}^p  \leq  |\alpha_n|^p\|G_{n,j}\|_{\mathrm{m}_p}^p+|\alpha_{n+j}|^p\|H_{n,j}\|_{\mathrm{m}_p}^p 
\leq |\alpha_n|^p\|\phi^{[1]}\|_{\mathrm{m}_p}^p+|\alpha_{n+j}|^p\|\phi^{[1]}\|_{\mathrm{m}_p}^p,\quad j\neq0.
\end{equation*}
Substituting this into \eqref{fj 22 estimate}, we obtain
\begin{equation}\label{fj jneq0}
\|F_j\|_{\mathrm{m}_p}\lesssim \|\phi^{[1]}\|_{\mathrm{m}_p}\|\alpha\|_{\ell_{p^{\sharp}}},\quad j\neq0.
\end{equation}

For $j=0,$ we have $F_{n,0}=\alpha_n G_{n,n}.$ Thus,
\begin{equation}\label{fj jeq0}
\|F_0\|_{\mathrm{m}_p}\leq\|\alpha\|_{\ell_{p^{\sharp}}}\|\phi^{[1]}\|_{\mathrm{m}_p}.
\end{equation}

Substituting \eqref{fj jneq0} and \eqref{fj jeq0} into \eqref{ar preliminary estimate}, we obtain
$$\|A_R\|_{\mathrm{m}_p}\lesssim\|\alpha\|_{\ell_{p^{\sharp}}}\|\phi^{[1]}\|_{\mathrm{m}_p}.$$

Denote by $\phi_{1}$ the function $\phi_{\alpha}$ when the sequence $\alpha$ consists of $1$'s. Since $\phi$ is supported in $(0,1)$, we have that 
$$
\phi_{\alpha} = \phi_{1}\sum_{j\in \mathbb{Z}} \alpha_j\chi_{[j,j+1)}.
$$
Recalling the notation $G_{\alpha}$ and $H_{\alpha}$ from Lemma \ref{GH lemma}, the second summand in \eqref{diagonal_splitting_2}, is expressed as
\begin{align*}
B_R(t,s)&=\phi_{\alpha}(t)\cdot \frac{\chi_{|\lfloor s\rfloor-\lfloor t\rfloor|> R}}{t-s} -\phi_{\alpha}(s)\cdot \frac{\chi_{|\lfloor s\rfloor-\lfloor t\rfloor|> R}}{t-s }\\ &=
\phi_1(t)\cdot G_{\alpha}(t,s)-\phi_1(s)\cdot H_{\alpha}(t,s).
\end{align*}
It follows that
$$\|B_R\|_{\mathrm{m}_p}^p\leq\|\{\phi_1(t)\}_{t,s\in\mathbb{R}}\|_{\mathrm{m}_p}^p\|G_{\alpha}\|_{\mathrm{m}_p}^p+\|\{\phi_1(s)\}_{t,s\in\mathbb{R}}\|_{\mathrm{m}_p}^p\|H_{\alpha}\|_{\mathrm{m}_p}^p.$$
It follows from Lemma \ref{GH lemma} and from trivial estimates
$$\|\{\phi_1(t)\}_{t,s\in\mathbb{R}}\|_{\mathrm{m}_p},\|\{\phi_1(s)\}_{t,s\in\mathbb{R}}\|_{\mathrm{m}_p}\leq\|\phi\|_{\infty}$$
that
$$\|B_R\|_{\mathrm{m}_p}\leq c_p\|\phi\|_{\infty}\|\alpha\|_{p^{\sharp}}.$$
Finally, \eqref{diagonal_splitting_2} yields the result.
\end{proof}

Using Theorem \ref{quasi_wavelet_bernstein}, we obtain the following analogy for $0<p<1$ of Lemma \ref{peller_fav_easy}.
\begin{lemma}\label{quasi_fav_lemma}
    Let $\psi$ be a compactly supported $C^\beta$-wavelet, where $\beta > \frac{2}{p}$, and let $f$ be a locally integrable function on $\mathbb{R}$ such that $f_j\in L_{p^\sharp}(\mathbb{R})$. Then
    \begin{equation*}
        \|f_j^{[1]}\|_{\mathrm{m}_p} \lesssim_p 2^{\frac{j}{p}}\|f_j\|_{p^\sharp}.
    \end{equation*}
\end{lemma}
\begin{proof}
    By definition \eqref{f_j_definition},
    \begin{equation*}
        f_j(t) = \sum_{k\in \mathbb{Z}} 2^{\frac{j}{2}}\psi(2^j t-k)\langle f, \psi_{j,k}\rangle,\quad t \in \mathbb{R}.
    \end{equation*}
    Theorem \ref{quasi_wavelet_bernstein} with $\phi=\psi$ yields
    \begin{equation*}
        \|f_j^{[1]}\|_{\mathrm{m}_p} \lesssim 2^{\frac{3j}{2}}\left(\sum_{k\in \mathbb{Z}} |\langle f,\psi_{j,k}\rangle|^{p^\sharp}\right)^{1/p^\sharp}.
    \end{equation*}
    Now applying Lemma \ref{wavelet_bernstein} gives us
    \begin{equation*}
        2^{\frac{3j}{2}}\left(\sum_{k\in \mathbb{Z}} |\langle f,\psi_{j,k}\rangle|^{p^\sharp}\right)^{1/p^\sharp}\lesssim 2^{\frac{3j}{2}}\cdot 2^{j\left(\frac{1}{p^{\sharp}}-\frac{1}{2}\right)}\|f_j\|_{p^\sharp}  = 2^{\frac{j}{p}}\|f_j\|_{p^\sharp}.
    \end{equation*}
\end{proof}

Lemma \ref{quasi_fav_lemma} gives us the following result, which generalises Corollary \ref{Peller_critereon} and is proved in the same way.
\begin{theorem}
    Let $0 < p \leq 1$. Then
    \begin{equation*}
        \|f^{[1]}\|_{\mathrm{m}_p} \lesssim \|f'\|_\infty+\left(\sum_{j\in \mathbb{Z}} 2^{j}\|f_j\|_{p^\sharp}^p\right)^{1/p} = \|f'\|_\infty+\|f\|_{\dot{B}^{\frac{1}{p}}_{p^\sharp,p}}
    \end{equation*}
    for all Lipschitz functions $f$ belonging to $\dot{B}^{\frac{1}{p}}_{p^\sharp,p}(\mathbb{R})$.
\end{theorem}
\begin{proof}
    By Lemma \ref{besov_realisation}, there exists a constant $c$ such that
    \begin{equation*}
        f(t) = f(0)+ct+\sum_{j\in \mathbb{Z}} f_j(t)-f_j(0),\quad t \in \mathbb{R}.
    \end{equation*}
    Therefore,
    \begin{equation*}
        f^{[1]} = c+\sum_{j\in \mathbb{Z}} f_j^{[1]}.
    \end{equation*}
    Using the $p$-triangle inequality \eqref{p-prop}, it follows that
    \begin{equation*}
        \|f^{[1]}\|_{\mathrm{m}_p}^p \leq |c|^p + \sum_{j\in \mathbb{Z}} \|f_j^{[1]}\|_{\mathrm{m}_p}^p.
    \end{equation*}
    Bounding the $j$th summand with Lemma \ref{quasi_fav_lemma}, 
    \begin{equation*}
        \|f^{[1]}\|_{\mathrm{m}_p}^p \lesssim |c|^p + \sum_{j\in \mathbb{Z}} 2^{\frac{j}{p}}\|f_j\|_{p^{\sharp}}^p = |c|^p+\|f\|_{\dot{B}^\frac{1}{p}_{p^\sharp,p}}^p.
    \end{equation*}
    Since $|c|\lesssim \|f'\|_\infty+\|f\|_{B^{\frac{1}{p}}_{p^\sharp,p}}$, the result follows.
\end{proof}
This completes the proof of Theorem \ref{main_sufficient_condition}.

\begin{remark}
    Using Theorem \ref{main_sufficient_condition}, it is possible to extend \eqref{peller_fav_thm} to the range $0<p<1$ in a certain sense. That is, if $f\in L_{p^\sharp}(\mathbb{R})$ is a distribution with Fourier transform 
    supported in the set $[-\frac{12}{7}\sigma,-\sigma)\cup (\sigma,\frac{12}{7}\sigma]$ where $\sigma>0$, then
    \begin{equation*}
        \|f^{[1]}\|_{\mathrm{m}_p} \lesssim_p \sigma^{\frac{1}{p}}\|f\|_{p^\sharp}.
    \end{equation*}
    Note that by rescaling if necessary and applying \eqref{schur_homogeneity} it suffices to take $\sigma=1$, so that $\Delta_0 f = f$ and $\Delta_{n}f=0$ for $n\neq 0$. Assume 
    now that $f$ is a $p^\sharp$-integrable distribution with Fourier transform supported in $[-\frac{12}{7},-1)\cup (1,\frac{12}{7}]$.

    It follows from Theorem \ref{main_sufficient_condition} and Bernstein's inequality \cite[Corollary 1.5]{Sawano2018} that
    \begin{equation*}
        \|f^{[1]}\|_{\mathrm{m}_p} \lesssim_p \|f'\|_{\infty}+\|f\|_{\dot{B}^{\frac{1}{p}}_{p^\sharp,p}} \lesssim_p \|f\|_{\infty}+\|f\|_{\dot{B}^{\frac{1}{p}}_{p^\sharp,p}}.
    \end{equation*}
    According to \cite[Corollary 1.8]{Sawano2018}, we have $\|f\|_{\infty} \lesssim_p \|f\|_{p^\sharp}$ and hence
    \begin{equation*}
        \|f^{[1]}\|_{\mathrm{m}_p} \lesssim_p \|f\|_{p^\sharp}+\|f\|_{\dot{B}^{\frac{1}{p}}_{p^\sharp,p}}.
    \end{equation*}
    Using the definition of the Besov semi-norm \eqref{besov_seminorm_def}, that $\Delta_0 f = f\in L_{p^\sharp}(\mathbb{R})$ and that $\Delta_nf=0$ for $n\neq 0$, we have
    \begin{align*}
        \|f\|_{\dot{B}^{\frac{1}{p}}_{p^\sharp,p}} = \left(\sum_{n\in \mathbb{Z}} 2^{n}\|\Delta_n f\|_{p^\sharp}^p\right)^{\frac{1}{p}} = \|\Delta_0 f\|_{p^\sharp} = \|f\|_{p^\sharp}.
    \end{align*}
    Hence, for $0< p\leq 1$ and $f$ with Fourier transform supported in $[-\frac{12}{7},-1)\cup (1,\frac{12}{7}]$ we have
    \begin{equation*}
        \|f^{[1]}\|_{\mathrm{m}_p} \lesssim_p \|f\|_{p^\sharp}.
    \end{equation*}
\end{remark}
\subsection{Submajorisation inequalities}\label{submajor_section}
In this section we assume that $0 < p \leq 1$, and $\psi$ is a $C^\beta$ compactly supported wavelet where $\beta > \frac{2}{p}$. All wavelet components $f_j$
are computed with respect to $\psi$.

In terms of singular values, Theorem \ref{main_sufficient_condition} states that there exists a constant $C_p > 0$ such that for all self-adjoint bounded operators $A$ and $B$
with $A-B\in \mathcal{L}_p$ we have
\begin{equation*}
    \sum_{k=0}^\infty \mu(k,f(A)-f(B))^p \leq C_p^p(\|f'\|_\infty+\|f\|_{B^{\frac{1}{p}}_{p^\sharp,p}})^p\sum_{k=0}^\infty \mu(k,A-B)^p.
\end{equation*}
Using a short argument borrowed from \cite{HSZ2019} we can strengthen this inequality. For all bounded self-adjoint operators $A$ and $B$ with $A-B$ compact, in this section we will prove that the following stronger statement holds:
\begin{equation*}
    \sum_{k=0}^n \mu(k,f(A)-f(B))^p \leq K_p^p(\|f'\|_\infty+\|f\|_{B^{\frac{1}{p}}_{p^\sharp,p}})^p\sum_{k=0}^n \mu(k,A-B)^p,\quad n\geq 0.
\end{equation*}
Here $K_p>0$ is a constant. In principle it may be that $K_p$ is larger than $C_p$, but $K_p$ is independent of $n$. The argument in \cite{HSZ2019} is based on real interpolation of the couple $(\mathcal{L}_p,\mathcal{L}_{\infty})$, we recall
the relevant details in a mostly self-contained manner here.

We make use of the following inequality originally due to Rotfel'd \cite{Rotfeld1968}, which holds for $0 < p \leq 1$
and compact operators $X$ and $Y$,
\begin{equation}\label{p_ky_fan}
    \mu(X+Y)^p \prec\prec \mu(X)^p+\mu(Y)^p.
\end{equation} 
Here, $\prec\prec$ denotes submajorisation in the sense of Hardy, Littlewood and P\'{o}lya. The meaning of \eqref{p_ky_fan} is that for all $n\geq 0$ we have
\begin{equation*}
    \sum_{k=0}^n \mu(k,X+Y)^p \leq \sum_{k=0}^n \mu(k,X)^p+\mu(k,Y)^p.
\end{equation*}
An alternative perspective on \eqref{p_ky_fan} is that it follows from the fact that $t\mapsto t^p$ is operator monotone when $0 < p\leq 1$. See \cite[Theorem 3.7]{DS-concave-2009}.

We will make use of the following lemma, which is purely operator theoretic.
\begin{lemma}\label{infimum_lemma}
    Let $X$ be a compact operator. For all $n\geq 0$ there exists a projection
    $P$ such that
    \begin{equation*}
        \|X(1-P)\|_p^p+(n+1)\|XP\|_{\infty}^p \leq 2\sum_{k=0}^n \mu(k,X)^p.
    \end{equation*}
    The projection $P$ can be chosen such that $1-P$ has finite rank.
\end{lemma}
\begin{proof}
    Passing to a polar decomposition $|X| = UX$ if necessary, it suffices to take $X\geq 0$.
    Let $P$ denote the projection
    \begin{equation*}
        P := \chi_{[0,\mu(n,X))}(X).
    \end{equation*}
    Then,
    \begin{equation*}
        \mu(j,X(1-P)) = \mu(j,X),\quad j\leq n,\quad \mu(n+1,X(1-P)) = 0.
    \end{equation*}
    It follows that
    \begin{equation}\label{lower_cut}
        \|X(1-P)\|_p^p = \sum_{j=0}^n \mu(j,X(1-P))^p = \sum_{j=0}^n \mu(j,X)^p 
    \end{equation}
    and
    \begin{equation*}
        \|XP\|_{\infty} \leq \mu(n,X).
    \end{equation*}
    Therefore,
    \begin{equation}\label{upper_cut}
        (n+1)\|XP\|_{\infty}^p \leq (n+1)\mu(n,X)^p \leq \sum_{j=0}^n \mu(j,X)^p.
    \end{equation}
    Adding \eqref{lower_cut} and \eqref{upper_cut} yields the result.    
\end{proof}

From Theorem \ref{main_sufficient_condition} we can deduce the following result, whose proof is based on Theorem 6.1 in \cite{HSZ2019}. Recall that $\psi$ is a fixed compactly supported $C^\beta$ wavelet, where $\beta > \frac{2}{p}$
and for $j\in \mathbb{Z}$ we denote
\begin{equation*}
    f_j = \sum_{k\in \mathbb{Z}} \psi_{j,k}\langle f,\psi_{j,k}\rangle.
\end{equation*}
Since $\|\cdot\|_{\mathrm{m}_1}=\|\cdot\|_{m_\infty},$ it follows from Lemma \ref{peller_fav_easy} that
\begin{equation}\label{peller_fav_l_infty}
    \|f_j(A)-f_j(B)\|_{\infty} \lesssim 2^{j}\|f_j\|_{\infty}\|A-B\|_{\infty},\quad A,B\in \mathcal{B}_{\mathrm{sa}}(H).
\end{equation}    
\begin{theorem}\label{submajor_thm}
    Let $f$ be a locally integrable function on $\mathbb{R}$, and let $j\in \mathbb{Z}$ be such that $f_j\in L_\infty(\mathbb{R})$. There exists a constant $C_p>0$ such that for all bounded self-adjoint operators $A$ and $B$ with $A-B$ compact we have
    \begin{equation*}
        \mu(f_j(A)-f_j(B))^p \prec\prec C_p2^{j}\|f_j\|_{p^\sharp}^p\mu(A-B)^p.
    \end{equation*}
\end{theorem}
\begin{proof}
    Here, one uses the inequality
    $$\sum_{k=0}^n\mu^p(k,X+Y)\leq\|X\|_p^p+(n+1)\|Y\|_{\infty}^p,$$
    for compact $X$ and $Y$, which follows from \eqref{p_ky_fan}.

    Let $P$ be a projection with $1-P$ finite rank, and let $A_P = B+(A-B)P$. Then
    \begin{equation*}
        f_j(A)-f_j(B) = f_j(A)-f_j(A_P)+f_j(A_P)-f_j(B)
    \end{equation*}
    Therefore, for all $n\geq 0$ we have
    \begin{equation*}
        \sum_{k=0}^n \mu(k,f_j(A)-f_j(B))^p \leq \|f_j(A)-f_j(A_P)\|_p^p+(n+1)\|f_j(A_P)-f_j(B)\|_\infty^p.
    \end{equation*}
    Note that $A-A_P = (A-B)(1-P)$ and $A_P-B = (A-B)P$. Since $1-P$ has finite
    rank, $A-A_P\in \mathcal{L}_p$.
    Applying Theorem \ref{quasi_fav_lemma} to $\|f_j(A)-f_j(A_P)\|_p^p$ and \eqref{peller_fav_l_infty} for $\|f_j(A_P)-f_j(B)\|_\infty^p$ yields
    \begin{equation*}
        \sum_{k=0}^n \mu(k,f_j(A)-f_j(B))^p \lesssim_p 2^{j}\|f_j\|_{p^\sharp}^p\|A-A_P\|_p^p+(n+1)2^{jp}\|f_j\|_\infty^p\|A_P-B\|_\infty^p
    \end{equation*}
    where the constant is independent of $n$.
    From \eqref{sequential_bernstein}, we have that $\|f_j\|_\infty \lesssim_p 2^{j(\frac{1}{p}-1)} \|f_j\|_{p^\sharp}$. Therefore,
    \begin{equation*}
        \sum_{k=0}^n \mu(k,f_j(A)-f_j(B))^p \lesssim_p 2^{j}\|f_j\|_{p^\sharp}^p(\|A-A_P\|_p^p+(n+1)\|A_P-B\|_\infty^p).
    \end{equation*}
    Using Lemma \ref{infimum_lemma} with $X = A-B$ implies that there exists $P$
    such that
    \begin{equation*}
        \|A-A_P\|_p^p+(n+1)\|A_P-B\|_\infty^p = \|(A-B)(1-P)\|_{p}^p+(n+1)\|(A-B)P\|_\infty^p \leq 2\sum_{j=0}^n \mu(j,A-B)^p.
    \end{equation*}
    Thus,
    \begin{equation*}
        \sum_{k=0}^n \mu(k,f_j(A)-f_j(B))^p \lesssim_p 2^{j}\|f_j\|_{p^\sharp}^p\sum_{k=0}^n \mu(k,A-B)^p,\quad n\geq 0
    \end{equation*}
    as required.
\end{proof}

Representing $f$ as $ct+\sum_{j\in \mathbb{Z}}f_j-f_j(0)$ and applying \eqref{p_ky_fan}, we arrive at the following submajorisation result.
\begin{corollary}
    Let $0 < p \leq 1$. For Lipschitz functions $f \in \dot{B}^{\frac{1}{p}}_{p^\sharp,p}(\mathbb{R})$ and all bounded self-adjoint operators $A$ and $B$ with $A-B$ compact we have
    $$
        |f(A)-f(B)|^p \prec\prec C_{p}(\|f'\|_\infty+\|f\|_{\dot{B}^{\frac{1}{p}}_{p^\sharp,p}})^p|A-B|^p.
    $$
    Equivalently, for any fully symmetric operator space $\mathcal{E}$ (see e.g. \cite[Definition 2.5.7]{LSZ}), we have the Lipschitz estimate
    $$
        \||f(A)-f(B)|^p\|_{\mathcal{E}} \lesssim (\|f'\|_\infty+\|f\|_{\dot{B}^{\frac{1}{p}}_{p^\sharp,p}})^p\||A-B|^p\|_{\mathcal{E}}.
    $$
\end{corollary}

\subsection{$\mathcal{L}_p$-operator H\"older functions}\label{holder_section}
To complement Theorem \ref{main_sufficient_condition}, we study the related issue of operator H\"older estimates. It is well-known that all $\alpha$-H\"older functions belong to $\dot{B}^{\alpha}_{\infty,\infty}.$
The arguments in this section are inspired by those of Aleksandrov and Peller \cite[Section 5]{Aleksandrov-Peller-holder-2010}, with adaptations for the wavelet decomposition.    
We will take $0 < p \leq 1,$ and $\psi$ denotes a compactly supported $C^k$ wavelet where $k > \frac{2}{p}.$ Recall that if $f$ is a locally integrable function, we denote
\[
    f_j := \sum_{k\in \mathbb{Z}} \psi_{j,k}\langle f,\psi_{j,k}\rangle.
\]   
This series is finite on bounded subsets of $\mathbb{R}.$

The following lemma concerns the representation of a H\"older continuous function $f$ by a wavelet series. This issue is parallel to Lemma \ref{besov_realisation},
and the proof is similar.
\begin{lemma}\label{holder_wavelet_realisation}
    Let $f$ be a H\"older continuous function on $\mathbb{R}$ of order $\alpha \in (0,1).$ Then
    \[
        f(t) = f(0)+\sum_{j\in \mathbb{Z}} f_j(t)-f_j(0),\quad t \in \mathbb{R}
    \]
    where the series converges uniformly on compact subsets of $\mathbb{R}.$
\end{lemma}
\begin{proof}
    Since $f$ is H\"older continuous of order $\alpha,$ we have $f \in \dot{B}^{\alpha}_{\infty,\infty}(\mathbb{R})$ and hence Theorem \ref{besov_space_wavelet_characterisation} implies
    \[
        \|f\|_{\dot{B}^{\alpha}_{\infty,\infty}} := \sup_{j\in \mathbb{Z}} 2^{j\alpha}\|f_j\|_{\infty}  < \infty.
    \]    
    Hence,
    \[
        |f_j(t)-f_j(0)| \leq 2\|f_j\|_{\infty} \leq 2^{1-j\alpha}\|f\|_{\dot{B}^{\alpha}_{\infty,\infty}}
    \]
    and hence the series
    \[
        \sum_{j\geq 0} f_j(t)-f_j(0)
    \]
    converges uniformly over $t \in \mathbb{R}$.
    
    For $K\geq 0,$ we have
    \[
        \sup_{-K\leq t \leq K}|f_j(t)-f_j(0)| \leq \sup_{-K\leq t\leq K}|t|\|f_j'\|_{\infty} = K\|f_j'\|_{\infty}.
    \]
    Applying the Bernstein-type inequality $\|f_j'\|_{\infty}\lesssim 2^{j}\|f_j\|_{\infty}$ \cite[Chapter 2, Theorem 3]{Meyer-wavelets-1992}, we arrive at
    \[
        \sup_{-K\leq t\leq K}|f_j(t)-f_j(0)| \lesssim K2^{j}\|f_j\|_{\infty}\leq K2^{j(1-\alpha)}\|f\|_{\dot{B}^{\alpha}_{\infty,\infty}}.
    \]
    Thus, the series
    \[
        \sum_{j<0} f_j(t)-f_j(0)
    \]
    converges uniformly over $-K\leq t \leq K.$ Since $K$ is arbitrary, we have that the series
    \[
        \sum_{j\in \mathbb{Z}} f_j(t)-f_j(0)
    \]
    converges uniformly over compact subsets of $\mathbb{R}.$ Next we prove that $\sum_{j\in \mathbb{Z}} f_j-f_j(0)$ is an $\alpha$-H\"older continuous function. 
    For all $t,s\in \mathbb{R},$ we have
    \begin{align*}
        \left|\sum_{j\in \mathbb{Z}} f_j(t)-f_j(s)\right| &\leq \sum_{2^{j}<|t-s|} |f_j(t)-f_j(s)| + \sum_{2^j\geq |t-s|}|f_j(t)-f_j(s)|\\
                                                        &\leq \sum_{2^j<|t-s|} |t-s|\|f_j'\|_{\infty}+\sum_{2^{j}\geq |t-s|} 2\|f_j\|_{\infty}\\
                                                        &\leq \sum_{2^j<|t-s|} |t-s|2^{j(1-\alpha)}\|f\|_{\dot{B}^{\alpha}_{\infty,\infty}} + 2\sum_{2^j \geq |t-s|} 2^{-j\alpha}\|f\|_{\dot{B}^{\alpha}_{\infty,\infty}}\\
                                                        &\lesssim |t-s|^{\alpha}\|f\|_{\dot{B}^{\alpha}_{\infty,\infty}}.
    \end{align*}
    That is, the function $\sum_{j\in \mathbb{Z}} f_j(t)-f_j(0)$ is H\"older continuous of order $\alpha.$
    Now we consider the difference
    \[
        f-\sum_{j\in \mathbb{Z}} f_j-f_j(0).
    \]
    This must be a polynomial, since all of its wavelet coefficients vanish (see the discussion preceding the proof of Lemma \ref{besov_realisation}).
    Since $f$ and $\sum_{j\in \mathbb{Z}} f_j-f_j(0)$ are $\alpha$-H\"older for $\alpha<1$, it follows that $f-\sum_{j\in \mathbb{Z}} f_j-f_j(0)$ is a polynomial
    of degree $0$, i.e. a constant. Hence there exists $c \in \mathbb{C}$ such that
    \[
        f(t) = c+\sum_{j\in \mathbb{Z}} f_j(t)-f_j(0).
    \]
    Substituting $t=0$ yields $c = f(0).$
\end{proof}

\begin{theorem}\label{main_holder_trick}
    Let $j\in \mathbb{Z}$ and let $f$ be a locally integrable function such that $f_j$ is bounded. Then for all $\alpha \in (0,1)$, $0 < p \leq 1$ we have
    \begin{equation*}
        \|f_j(A)-f_j(B)\|_{\frac{p}{\alpha}} \lesssim_{p,\alpha} 2^{j(\alpha+\frac{1}{p^\sharp})}\|f_j\|_{p^\sharp}\|A-B\|_p^\alpha
    \end{equation*}
    for all bounded operators $A$ and $B$ such that $A-B \in \mathcal{L}_p$.
\end{theorem}
\begin{proof}
    For each $j\in \mathbb{Z}$, we have
    \begin{equation*}
        \|f_j(A)-f_j(B)\|_{\frac{p}{\alpha}} = \||f_j(A)-f_j(B)|^{1/\alpha}\|_{p}^{\alpha} \leq \|f_j(A)-f_j(B)\|_\infty^{1-\alpha}\|f_j(A)-f_j(B)\|_{p}^{\alpha}.
    \end{equation*}
    Since $f_j$ is bounded, the bounds $\|f_j(A)\|_\infty,\|f_j(B)\|_\infty\leq \|f_j\|_\infty$ give
    \begin{equation*}
        \|f_j(A)-f_j(B)\|_{\infty}^{1-\alpha} \leq 2^{1-\alpha}\|f_j\|_{\infty}^{1-\alpha}.
    \end{equation*}
    Now we apply \eqref{sequential_bernstein}, which gives us
    \begin{equation*}
        \|f_j(A)-f_j(B)\|_{\infty}^{1-\alpha} \leq 2^{1-\alpha}\|f_j\|_{\infty}^{1-\alpha} \lesssim_{\alpha,p} 2^{\frac{j(1-\alpha)}{p^\sharp}}\|f_j\|_{p^\sharp}^{1-\alpha}.
    \end{equation*}
    Using Lemma \ref{quasi_fav_lemma}, we also have
    \begin{equation*}
        \|f_j(A)-f_j(B)\|_p^{\alpha} \lesssim_p 2^{\frac{j\alpha}{p}}\|f_j\|_{p^\sharp}^{\alpha}\|A-B\|_p^{\alpha}.
    \end{equation*}
    Hence,
    \begin{equation*}
        \|f_j(A)-f_j(B)\|_{\frac{p}{\alpha}} \lesssim_{p,\alpha} 2^{\frac{j}{p^\sharp}}\|f_j\|_{p^\sharp}\|A-B\|_p^{\alpha}.
    \end{equation*}
\end{proof}

Theorem \ref{main_holder_trick} implies the following sufficient condition for a function to be $\alpha$-H\"older in $\mathcal{L}_p$. Interestingly, the condition differs
depending on $p\leq \alpha$ or $p> \alpha$. With $p=1$, this recovers \cite[Theorem 5.3]{Aleksandrov-Peller-holder-2010}.
\begin{theorem}
    Let $0 < p \leq 1$ and $\alpha \in (0,1)$. If $f$ is an $\alpha$-H\"older function such that $f \in \dot{B}^{\alpha+\frac{1}{p^\sharp}}_{p^\sharp,\min\{1,\frac{p}{\alpha}\}}(\mathbb{R})$
    then for all bounded self-adjoint operators $A$ and $B$ with $A-B\in \mathcal{L}_p$ we have
    \begin{equation*}
        \|f(A)-f(B)\|_{\frac{p}{\alpha}} \lesssim_{p,\alpha} \|f\|_{\dot{B}^{\alpha+\frac{1}{p^\sharp}}_{p^\sharp,\min\{1,\frac{p}{\alpha}\}}}\|A-B\|_p^{\alpha}.
    \end{equation*}
\end{theorem}
\begin{proof}
    With $\nu := \min\{1,\frac{p}{\alpha}\}$, the quasi-norm $\|\cdot \|_{\frac{p}{\alpha}}$ obeys a $\nu$-triangle inequality. Therefore, applying the representation from Lemma \ref{holder_wavelet_realisation} yields 
    \begin{equation*}
        \|f(A)-f(B)\|_{\frac{p}{\alpha}}^\nu \leq \sum_{j\in \mathbb{Z}} \|f_j(A)-f_j(B)\|_{\frac{p}{\alpha}}^\nu.
    \end{equation*}
    Theorems \ref{main_holder_trick} and \ref{besov_space_wavelet_characterisation} together imply
    \begin{equation*}
        \|f(A)-f(B)\|_{\frac{p}{\alpha}}^\nu \lesssim \sum_{j\in \mathbb{Z}} 2^{j\nu\left(\alpha+\frac{1}{p^\sharp}\right)}\|f_j\|_{p^\sharp}^\nu\|A-B\|_p^{\alpha \nu} \approx \|f\|_{\dot{B}^{\alpha+\frac{1}{p^\sharp}}_{p^\sharp,\nu}}^\nu \|A-B\|_p^{\alpha\nu}.
    \end{equation*}
\end{proof}

\subsection{Weak-type H\"older estimates}\label{weak_holder_section}
Aleksandrov and Peller have proved that for all $p\in [1,\infty]$, $\alpha \in (0,1)$ we have
\begin{equation*}
    \|f(A)-f(B)\|_{\frac{p}{\alpha},\infty} \lesssim_{p,\alpha} \|f\|_{C^{\alpha}}\|A-B\|_p^\alpha,\quad A=A^*,B=B^*\in \mathcal{B}(H),\,A-B\in \mathcal{L}_p
\end{equation*}
where $\|f\|_{C^\alpha}$ is the $\alpha$-H\"older norm (see \cite[Theorem 5.4]{Aleksandrov-Peller-holder-2010}). This result can be viewed as a complement to
the main result of \cite{cpsz}, which states that
\begin{equation*}
    \|f(A)-f(B)\|_{1,\infty} \lesssim \|f'\|_\infty \|A-B\|_1,\quad A=A^*,B=B^*\in \mathcal{B}(H),\,A-B\in \mathcal{L}_1.
\end{equation*}
In order to continue this theme, we will study H\"older-type estimates for $\|f(A)-f(B)\|_{\frac{p}{\alpha},\infty}$ where $0 < p < 1$.
    
The following argument is closely based on \cite[Theorem 5.1]{Aleksandrov-Peller-holder-2010}, the essential difference is that we use the wavelet decomposition
in place of the Littlewood-Paley decomposition. Note that by Theorem \ref{besov_space_wavelet_characterisation}, 
if $f \in \dot{B}^s_{p,q}(\mathbb{R})$ for some $s\in \mathbb{R}$, $p,q\in (0,\infty]$ then for every $j\in \mathbb{Z}$ we have $f_j \in L_{\infty}(\mathbb{R}).$
\begin{theorem}
    Let $\alpha \in (0,1)$ and $p \in (0,1]$. Let $f$ be an $\alpha$-H\"older function. Let $A$ and $B$ be self-adjoint bounded operators such that $A-B$ is compact. For all $n\geq 0$ we have
    \begin{equation}
        \mu(n,f(A)-f(B)) \lesssim_{p,\alpha} (1+n)^{-\frac{\alpha}{p}}\|f\|_{\dot{B}^{\alpha+\frac{1}{p^\sharp}}_{p^\sharp,\infty}}\left(\sum_{k=0}^n \mu(k,A-B)^p\right)^{\frac{\alpha}{p}}. 
    \end{equation}
\end{theorem}
\begin{proof}
    Let $N \in \mathbb{Z}$, to be specified shortly. By the inequality \eqref{p_ky_fan}, for all $n\geq 0$ we have
    \begin{equation*}
        \sum_{k=0}^n \mu(k,\sum_{j\leq N} f_j(A)-f_j(B))^p \leq \sum_{j\leq N} \sum_{k=0}^n \mu(k,f_j(A)-f_j(B))^p 
    \end{equation*}    
    According to Theorem \ref{submajor_thm}, we have
    \begin{equation*}
        \sum_{k=0}^n \mu(k,f_j(A)-f_j(B))^p \lesssim_p 2^{j}\|f_j\|_{p^\sharp}^p\sum_{k=0}^n \mu(k,A-B)^p.
    \end{equation*}
    Therefore,
    \begin{equation*}
        \sum_{k=0}^n \mu(k,\sum_{j\leq N} f_j(A)-f_j(B))^p \lesssim_p \sum_{j\leq N} 2^{j}\|f_j\|_{p^\sharp}^p\left(\sum_{k=0}^n \mu(k,A-B)^p\right).
    \end{equation*}
    By Theorem \ref{besov_space_wavelet_characterisation}, for all $j \in \mathbb{Z}$ we have
    \begin{equation*}
        \|f_j\|_{p^\sharp} \lesssim 2^{-j(\alpha+\frac{1}{p^\sharp})}\|f\|_{\dot{B}^{\frac{1}{p^\sharp}+\alpha}_{p^\sharp,\infty}}.
    \end{equation*}
    Hence, taking into account that $1-\alpha p -\frac{p}{p^\sharp} = p(1-\alpha)>0$, we have
    \begin{align*}
        \sum_{k=0}^n \mu(k,\sum_{j\leq N} f_j(A)-f_j(B))^p &\lesssim_p \sum_{j\leq N} 2^{j(1-\alpha p-\frac{p}{p^\sharp})}\|f\|_{\dot{B}^{\frac{1}{p^\sharp}+\alpha}_{p^\sharp,\infty}}^p\left(\sum_{k=0}^n \mu(k,A-B)^p\right)\\
        &\lesssim_{p,\alpha} 2^{Np(1-\alpha)}\|f\|_{\dot{B}^{\frac{1}{p^\sharp}+\alpha}_{p^\sharp,\infty}}^p\left(\sum_{k=0}^n \mu(k,A-B)^p\right).
    \end{align*}
    That is,
    \begin{equation*}
        \left(\sum_{k=0}^n \mu(k,\sum_{j\leq N} f_j(A)-f_j(B))^p\right)^{\frac{1}{p}} \lesssim_{p,\alpha} 2^{N(1-\alpha)}\|f\|_{\dot{B}^{\frac{1}{p^\sharp}+\alpha}_{p^\sharp,\infty}}\left(\sum_{k=0}^n \mu(k,A-B)^p\right)^{\frac{1}{p}}.
    \end{equation*}
    It follows that
    \begin{equation}\label{low_frequency}
        \mu(n,\sum_{j\leq N} f_j(A)-f_j(B))\lesssim_{p,\alpha} (1+n)^{-\frac{1}{p}}2^{N(1-\alpha)}\|f\|_{\dot{B}^{\alpha+\frac{1}{p^\sharp}}_{p^\sharp,\infty}}\left(\sum_{k=0}^n \mu(k,A-B)^p\right)^{\frac{1}{p}}.
    \end{equation} 
    Putting this aside for the moment, we consider the norm $\left\|\sum_{j> N} f_j(A)-f_j(B)\right\|_\infty$. By the triangle inequality, this is controlled by
    \begin{equation*}
        \left\|\sum_{j> N} f_j(A)-f_j(B)\right\|_\infty \leq \sum_{j> N} \|f_j(A)-f_j(B)\|_\infty \leq \sum_{j> N} 2\|f_j\|_\infty.
    \end{equation*}
    By Theorem \ref{besov_space_wavelet_characterisation}, we have $\|f_j\|_\infty\lesssim_{\alpha} 2^{-j\alpha} \|f\|_{\dot{B}^{\alpha}_{\infty,\infty}}$. Therefore
    \begin{equation}\label{high_frequnecy}
        \left\|\sum_{j>N} f_j(A)-f_j(B)\right\|_\infty \lesssim_{\alpha} \sum_{j>N} 2^{-j\alpha}\|f\|_{\dot{B}^{\alpha}_{\infty,\infty}} \lesssim_{\alpha} 2^{-N\alpha}\|f\|_{\dot{B}^{\alpha}_{\infty,\infty}}.
    \end{equation}
    Now we combine \eqref{low_frequency} and \eqref{high_frequnecy} to estimate $\mu(n,f(A)-f(B)).$  Using the representation from Lemma \ref{holder_wavelet_realisation}, we have
    \[
        f(A)-f(B) = \sum_{j\in \mathbb{Z}} f_j(A)-f_j(B).
    \]
    Since $A$ and $B$ are bounded and the series in Lemma \ref{holder_wavelet_realisation} converges uniformly over compact subsets, this series converges in the operator norm.
    
    Therefore,
    \begin{align*}
        \mu(n,f(A)-f(B)) &= \mu\left(n,\sum_{j\leq N} f_j(A)-f_j(B) + \sum_{j> N} f_j(A)-f_j(B)\right)\\
                            &\leq \mu\left(n,\sum_{j\leq N} f_j(A)-f_j(B)\right)+\left\|\sum_{j> N} f_j(A)-f_j(B)\right\|_\infty\\
                            &\stackrel{\eqref{high_frequnecy}}{\lesssim_{\alpha}} \mu\left(n,\sum_{j\leq N} f_j(A)-f_j(B)\right)+ 2^{-N\alpha}\|f\|_{\dot{B}^{\alpha}_{\infty,\infty}}\\
                            &\stackrel{\eqref{low_frequency}}{\lesssim_{\alpha}} (1+n)^{-\frac{1}{p}}2^{N(1-\alpha)}\|f\|_{\dot{B}^{\alpha+\frac{1}{p^\sharp}}_{p^\sharp,\infty}}\left(\sum_{k=0}^n\mu(k,A-B)^p\right)^{1/p} + 2^{-N\alpha}\|f\|_{\dot{B}^{\alpha}_{\infty,\infty}}.
    \end{align*} 
    Now we choose $N\in \mathbb{Z}$ such that
    \begin{equation*}
        2^{-N-1} \leq (1+n)^{-\frac{1}{p}}\left(\sum_{k=0}^n \mu(k,A-B)^p\right)^{\frac{1}{p}} \leq 2^{-N}.
    \end{equation*}
    Hence,
    \begin{equation*}
        \mu(n,f(A)-f(B)) \lesssim_{\alpha,p} (\|f\|_{\dot{B}^{\frac{1}{p^\sharp}+\alpha}_{p^\sharp,\infty}}+\|f\|_{\dot{B}^{\alpha}_{\infty,\infty}}) ((1+n)^{-\frac{1}{p}+(1-\alpha)\frac{1}{p}}+(1+n)^{-\frac{\alpha}{p}})\left(\sum_{k=0}^n \mu(k,A-B)^p\right)^{\frac{\alpha}{p}}. 
    \end{equation*}
    Since $\dot{B}^{\frac{1}{p^\sharp}+\alpha}_{p^\sharp,\infty} \subseteq \dot{B}^{\alpha}_{\infty,\infty}$ (this follows from \eqref{sequential_bernstein}), the desired result follows.
\end{proof}
It follows immediately from the definition of the $\mathcal{L}_{\frac{p}{\alpha},\infty}$ quasi-norm \eqref{weak_lp_def} that we have the following:
\begin{theorem}\label{weak_holder_thm}
    Let $p \in (0,1]$ and $\alpha \in (0,1)$. Assume that $f \in \dot{B}^{\frac{1}{p^\sharp}+\alpha}_{p^\sharp,\infty}(\mathbb{R})$ is an $\alpha$-H\"older function. For all self-adjoint bounded operators $A$ and $B$ such that $A-B\in \mathcal{L}_p$ we have
    \begin{equation*}
        \|f(A)-f(B)\|_{\frac{p}{\alpha},\infty} \lesssim_{p,\alpha} \|f\|_{\dot{B}^{\frac{1}{p^\sharp}+\alpha}_{p^\sharp,\infty}} \|A-B\|_p^{\alpha}. 
    \end{equation*}
\end{theorem}

Theorem \ref{weak_holder_thm} is related to results obtained in \cite{HSZ2019,Ricard-2018,Sobolev-jlms-2017}. The class $S_{d,\alpha}$
introduced by A.~V.~Sobolev \cite{Sobolev-jlms-2017} is the space of functions $f \in C^d(\mathbb{R} \setminus \{0\})$ such that 
$$
    |f^{(k)}(x)|\lesssim_k |x|^{\alpha-k},\quad 0\leq k\leq d,\; x \in \mathbb{R}\setminus \{0\}.
$$
It was proved in \cite[Theorem 1.2]{HSZ2019} that if $f \in S_{d,\alpha}$ where $d > \frac{1}{p}+2$ then 
\[
    \|f(A)-f(B)\|_{\frac{p}{\alpha}} \lesssim_{p,\alpha,f} \|A-B\|_{p}^{\alpha},\quad A,B\in \mathcal{B}_{\mathrm{sa}}(H),\; A-B\in \mathcal{L}_p.
\]
Theorem \ref{weak_holder_thm} complements that result in the sense that the class of functions is wider, but the $\mathcal{L}_{\frac{p}{\alpha}}$ estimate is weakened to an $\mathcal{L}_{\frac{p}{\alpha},\infty}$
estimate. 
We can prove that the class of functions is indeed larger using the following characterisation of the Besov seminorm. If $n\in \mathbb{N}$ is such that $\max\{\frac{1}{p^\sharp}-1,0\} < s < n$, then
\begin{equation}\label{holder_zygmund_characterisation}
    \|f\|_{\dot{B}^s_{p^{\sharp},\infty}} \approx \sup_{h>0} h^{-s}\left(\int_{-\infty}^\infty \left|\sum_{k=0}^n \binom{n}{k}(-1)^{n-k}f(t+kh)\right|^{p^{\sharp}}\,dt\right)^{1/p^\sharp}.  
\end{equation}
The choice of $n>s$ is irrelevant. This is well-known when $p^\sharp\geq 1$, see e.g. \cite[Theorem 2.39]{Sawano2018}. For $p^\sharp < 1,$ see \cite[Chapter 11, Theorem 18]{PeetreBook}.

\begin{theorem}
    If $d > \frac{1}{p}$ and $0 < \alpha < 1$, then
    \begin{equation*}
        S_{d,\alpha} \subset \dot{B}^{\alpha+\frac{1}{p^\sharp}}_{p^\sharp,\infty}(\mathbb{R}).
    \end{equation*}
\end{theorem}
\begin{proof}
    Let $f \in S_{d,\alpha}.$ Since $d > \frac{1}{p}$, in particular we have $d > \alpha+\frac{1}{p}-1 = \alpha+\frac{1}{p^\sharp}$. It follows that we can take $n=d$
    in \eqref{holder_zygmund_characterisation}. Therefore,
    \[
        \|f\|_{\dot{B}^{\frac{1}{p^\sharp}+\alpha}_{p^\sharp,\infty}} \approx \sup_{h>0} h^{-\frac{1}{p^\sharp}-\alpha}\left(\int_{-\infty}^\infty \left|\sum_{k=0}^d \binom{d}{k}(-1)^{d-k}f(t+kh)\right|^{p^{\sharp}}\,dt\right)^{1/p^\sharp}.  
    \]
    We split the integral into the regions $|t|<dh+h$ and $|t|>dh+h.$ That is,
    \begin{align*}
        \int_{-\infty}^\infty \left|\sum_{k=0}^d \binom{d}{k}(-1)^{d-k}f(t+kh)\right|^{p^{\sharp}}\,dt &= \int_{-dh-h}^{dh+h} \left|\sum_{k=0}^d \binom{d}{k}(-1)^{d-k}f(t+kh)\right|^{p^{\sharp}}\,dt\\
                                                                                                        &\quad+\int_{|t|>dh+h} \left|\sum_{k=0}^d \binom{d}{k}(-1)^{d-k}f(t+kh)\right|^{p^{\sharp}}\,dt.
    \end{align*}
    By the mean value theorem, we have
    \begin{equation*}
        \sum_{k=0}^d \binom{d}{k}(-1)^{d-k}f(t+kh) = (dh)^d\int_0^1 (1-\theta)^{d-1}f^{(d)}(t+dh\theta)\,d\theta.
    \end{equation*}
    So by the definition of the Sobolev class $S_{d,\alpha},$ when $t>dh+h$ we have 
    \[
        \left|\sum_{k=0}^d \binom{d}{k}(-1)^{d-k}f(t+kh)\right| \lesssim_d h^d \int_0^1 (1-\theta)^{d-1}|t+dh\theta|^{\alpha-d}\,d\theta \lesssim_d h^d|t|^{\alpha-d}
    \]
    and when $t < -dh-h$,
    \[
        \left|\sum_{k=0}^d \binom{d}{k}(-1)^{d-k}f(t+kh)\right| \lesssim_d h^d |t+dh|^{\alpha-d}.
    \]
    Since $\alpha<1$ and $d > \frac{1}{p},$ the function $|t|^{p^\sharp(\alpha-d)}$ is integrable over $[h,\infty),$ and therefore
    \[
        \int_{|t|>dh+h}\left|\sum_{k=0}^d \binom{d}{k}(-1)^{d-k}f(t+kh)\right|^{p^\sharp}\,dt \lesssim_d \int_{h}^\infty h^{dp^\sharp}|t|^{p^{\sharp}(\alpha-d)}\,dt \approx_{d,p} h^{\alpha p^\sharp+1}.
    \]
    On the other hand, for $|t|<dh+h$ we use the estimate
    \begin{equation*}
        \left|\sum_{k=0}^d \binom{d}{k}(-1)^{d-k}f(t+kh)\right| \lesssim_d \max_{0\leq \theta \leq dh} |t+\theta|^{\alpha} \leq (|t|+dh)^{\alpha} \lesssim_d |t|^{\alpha}+h^{\alpha}.
    \end{equation*}
    Therefore,
    \[
        \int_{|t|\leq dh+h}\left|\sum_{k=0}^d \binom{d}{k}(-1)^{d-k}f(t+kh)\right|^{p^\sharp}\,dt \lesssim_{d,p} \int_{|t|<dh+h} |t|^{\alpha p^\sharp}+h^{\alpha p^\sharp}\,dt \approx_{d,p} h^{\alpha p^{\sharp}+1}.
    \]
    It follows that for $h > 0$ we have
    \begin{equation*}
        h^{-\alpha-\frac{1}{p^\sharp}}\left(\int_{-\infty}^\infty \left|\sum_{k=0}^d \binom{d}{k}(-1)^{d-k}f(t+kh)\right|^{p^\sharp}\,dt\right)^{\frac{1}{p^\sharp}} \lesssim_{d,p} h^{\frac{1}{p^{\sharp}}+\alpha-\frac{1}{p^{\sharp}}-\alpha} = 1.
    \end{equation*}
    Taking the supremum over $h > 0$ yields $f \in \dot{B}^{\alpha+\frac{1}{p^\sharp}}_{p^\sharp,\infty}(\mathbb{R})$ from \eqref{holder_zygmund_characterisation}.
    
    The embedding is strict, because $S_{d,\alpha}$ is not closed under translation while $\dot{B}^{\alpha+\frac{1}{p^\sharp}}_{p^\sharp,\infty}(\mathbb{R})$ is.
\end{proof} 

\appendix

\section{Automatic complete boundedness of $\mathcal{L}_p$-bounded Schur multipliers}
The following is a recently published result of Aleksandrov and Peller \cite[Theorem 3.1]{Aleksandrov-Peller-acb-2020}.
\begin{theorem}\label{acb}
    Let $A \in M_n(\mathbb{C})$ be a matrix, where $1\leq n\leq \infty$. Let $N\geq 1$ and denote by $\mathrm{id}_{M_N(\mathbb{C})}$ the $N\times N$ matrix with all entries equal to $1$. Then
    for all $0 < p < 1$ we have
    \begin{equation*}
        \|A\|_{\mathrm{m}_p} = \|A\otimes \mathrm{id}_{M_{N}(\mathbb{C})}\|_{\mathrm{m}_p}.
    \end{equation*}
\end{theorem}
In other words, bounded Schur multipliers of $\mathcal{L}_p$ are automatically completely bounded.
The analogous statement for $p=1$ is well-known, see \cite[Theorem 5.1]{Pisier-book-2001}.
For the sake of completeness we include a proof of Theorem \ref{acb}. The proof is different from that of \cite{Aleksandrov-Peller-acb-2020}, and is
instead closely modelled on a proof for the $p=1$ case due to Smith \cite[Theorem 2.1]{Smith-cb-module-maps-jfa-1991}.

Recall that we denote by $\ell_2^n$ the $n$-dimensional Hilbert space.
\begin{proof}[Proof of Theorem \ref{acb}]
Let $1\leq n\leq\infty$ and $N\geq 1$. Let $u \in \ell_2^n\otimes \ell_2^N$ be a unit vector. Write the components of $u$ as $u =\sum_{j,l} u_{j,l}e_j\otimes e_l$. Consider the mapping:
\begin{equation*}
    Q_u:\ell_2^n\to \ell_2^n\otimes \ell_2^N
\end{equation*}
given by
\begin{equation*}
    Q_{u}e_j = \begin{cases} 
                    (\sum_{l=1}^N |u_{j,l}|^2)^{-1/2}\sum_{l=1}^N u_{j,l}e_j\otimes e_l,&\text{ if } \sum_{l=1}^N |u_{j,l}|^2 \neq 0\\
                    0,&\text{ otherwise.}
                \end{cases}
\end{equation*}

The adjoint of $Q_u$ is easily computed. We have
\begin{equation*}
    Q_u^*(e_{j}\otimes e_l) = e_j\overline{u}_{j,l}(\sum_{r=1}^N |u_{j,r}|^2)^{-1/2},\quad 1\leq j\leq n,\, 1\leq l\leq N
\end{equation*}
or $Q_u^*(e_j\otimes e_l) = 0$ if $\sum_{r=1}^N |u_{j,r}|^2 = 0$.

Then we compute $Q_u^*Q_ue_j$. If $\sum_{l=1}^N |u_{j,l}|^2 = 0$ then $Q_u^*Q_ue_j = 0$ and so assuming otherwise we have for every $1\leq j\leq n$,
\begin{align*}
    Q_u^*Q_ue_j &= \left(\sum_{l=1}^N |u_{j,l}|^2\right)^{-1/2}\sum_{r=1}^N u_{j,r}Q_u^*(e_{j}\otimes e_r)\\
                &= \left(\sum_{l=1}^N |u_{j,l}|^2\right)^{-1/2}\sum_{r=1}^N u_{j,r}\overline{u}_{j,r}(\sum_{l=1}^N |u_{j,l}|^2)^{-1/2}e_j\\
                &= e_j \cdot \frac{\sum_{r=1}^N |u_{j,r}|^2}{\sum_{r=1}^N |u_{j,r}|^2}\\
                &= e_j.
\end{align*}
So $Q_u$ is indeed a contraction.

Given $u \in \ell_2^n\otimes \ell_2^N$, define $\widetilde{u} \in \ell_{2}^n$ as,
\begin{equation*}
    \widetilde{u} = \sum_{j=1}^n (\sum_{l=1}^N |u_{j,l}|^2)^{1/2}e_j.
\end{equation*}
We have that $\|\widetilde{u}\|_{\ell_2^n} = \|u\|_{\ell_2^n\otimes \ell_2^N}$.

We now assert that
\begin{equation}\label{inflation}
    (\mathrm{id}_{M_N(\mathbb{C})}\otimes A)\circ (u\otimes v) = Q_u(A\circ \widetilde{u}\otimes \widetilde{v})Q_v^*.
\end{equation}
It suffices to check \eqref{inflation} entrywise. On the left hand side, we have
\begin{equation*}
    \langle e_{j,l_1},(\mathrm{id}_{M_N(\mathbb{C})}\otimes A)\circ (u\otimes v)e_{k,l_2}\rangle = A_{j,k}u_{j,l_1}\overline{v_{k,l_2}}, \quad 1\leq j\leq n, 1\leq k\leq n, 1\leq l_1,l_2\leq N
\end{equation*}
and on the right
\begin{align*}
    \langle e_{j,l_1},Q_u(A\circ (\widetilde{u}\otimes \widetilde{v}))Q_v^*e_{k,l_2}\rangle &= \langle Q_u^*e_{j,l_1},(A\circ \widetilde{u}\otimes \widetilde{v})Q_v^*e_{k,l_2}\rangle\\
                                                                                            &= u_{j,l_1}\overline{v_{k,l_2}}(\sum_{r=1}^N |u_{j,r}|^2)^{-1/2}(\sum_{r=1}^N |u_{k,r}|^2)^{-1/2}\langle e_j,(A\circ \widetilde{u}\otimes \widetilde{v})e_k\rangle\\
                                                                                            &= A_{j,k}u_{j,l_1}\overline{v_{k,l_2}}.
\end{align*}
This verifies \eqref{inflation}.

Since $Q_u$ and $Q_v$ are contractions, \eqref{inflation} implies that
\begin{equation*}
    \|(\mathrm{id}_{M_N(\mathbb{C})}\otimes A)\circ (u\otimes v)\|_{p} \leq \|A\|_{\mathrm{m}_p}\|u\|_{\ell_2^n\otimes \ell_2^N}\|v\|_{\ell_2^n\otimes \ell_2^N}.
\end{equation*}
Taking the supremum over $u,v \in \ell_2^n\otimes \ell_2^N$ with norm at most $1$ and using Lemma \ref{rank_one_suffices} yields the conclusion
\begin{equation*}
    \|\mathrm{id}_{M_N(\mathbb{C})}\otimes A\|_{\mathrm{m}_p} \leq \|A\|_{\mathrm{m}_p}.
\end{equation*}
The reverse inequality is clear.
\end{proof}

\end{document}